\documentclass[12pt,twoside]{article}
\usepackage{amsmath,amsfonts,amssymb,a4,enumitem,epsfig,array,psfrag}
\usepackage{url}
\usepackage{amsthm}
\usepackage{subfig}
\usepackage{chngcntr} 
\usepackage{color}
\usepackage[section]{placeins}
\usepackage[colorlinks=true]{hyperref}

\usepackage{algorithm}
\usepackage{algpseudocode}

\graphicspath{{figures/}}

\newtheorem{theo}{Theorem}

\newtheorem{prop}{Proposition}[section]
\newtheorem{coro}[prop]{Corollary}
\newtheorem{defin}[prop]{Definition}
\newtheorem{lemma}[prop]{Lemma}

\theoremstyle{definition}

\theoremstyle{remark}
\newtheorem{case}{Case}
\counterwithin*{case}{theo}
\counterwithin*{case}{prop}
\counterwithin*{case}{fact}

\newcommand{\half}{{\frac{1}{2}}}
\newcommand{\NN}{{\mathbb{N}}}
\newcommand{\ZZ}{{\mathbb{Z}}}

\newcommand{\RR}{{\mathbb{R}}}

\newcommand{\dist}{\operatorname{dist}}
\newcommand{\interior}{\operatorname{int}}
\newcommand{\cR}{\mathcal{R}}
\newcommand{\cD}{\mathcal{D}}
\newcommand{\cS}{\mathcal{S}}
\newcommand{\clS}{\bar{\cS}}
\newcommand{\cA}{\mathcal{A}}
\newcommand{\cB}{\mathcal{B}}
\newcommand{\Tw}{\operatorname{Tw}}
\newcommand{\TW}{\operatorname{TW}}

\newcommand{\ccol}{\operatorname{color}} 

\newcommand{\wind}{\operatorname{wind}}

\newcommand{\Aout}{\cA^{\vu}}

\newcommand{\Link}{\operatorname{Link}}
\newcommand{\Wr}{\operatorname{Wr}}
\newcommand{\ST}{T_{\operatorname{sym}}}
\newcommand{\bB}{\mathbf{B}}
\newcommand{\tv}{\vec{v}}
\newcommand{\vu}{\vec{u}}
\newcommand{\vw}{\vec{w}}
\newcommand{\vbeta}{\vec{\beta}}
\newcommand{\ex}{\vec{\mathbf{i}}}
\newcommand{\ey}{\vec{\mathbf{j}}}
\newcommand{\ez}{\vec{\mathbf{k}}}

\newcommand{\plshalf}[1]{{#1}^\sharp}
\newcommand{\base}{\vw}
\newcommand{\tbase}{t_{\base}}
\newcommand{\tbasez}{t_{\ez}}
\newcommand{\tbasew}{t_{\vw}}

\makeatletter
\newcommand{\subjclass}[2][2010]{%
  \let\@oldtitle\@title%
  \gdef\@title{\@oldtitle\footnotetext{#1 \emph{Mathematics Subject Classification.} #2}}%
}
\newcommand{\keywords}[1]{%
  \let\@@oldtitle\@title%
  \gdef\@title{\@@oldtitle\footnotetext{\emph{Key words and phrases.} #1.}}%
}
\makeatother

\begin{document}
\title{Domino tilings of three-dimensional regions:\\flips and twists}
\author{Pedro H. Milet \and Nicolau Saldanha} 

\subjclass{Primary 05B45, 52C22; Secondary 57M25, 05C70, 52C20.}
\keywords{Three-dimensional tilings, 
dominoes,
dimers,
flip accessibility,
connectivity by local moves,
writhe,
knot theory}
\maketitle

\begin{abstract}
In this paper, we consider domino tilings of regions of the form $\cD \times [0,n]$, where $\cD$ is a simply connected planar region and $n \in \NN$. It turns out that, in nontrivial examples, the set of such tilings is not connected by \emph{flips}, i.e., the local move performed by removing two adjacent dominoes and placing them back in another position.
We define an algebraic invariant, the \emph{twist}, which partially characterizes the connected components by flips of the space of tilings of such a region. 
Another local move, the \emph{trit}, consists of removing three adjacent dominoes, no two of them parallel, and placing them back in the only other possible position: performing a trit alters the twist by $\pm 1$.
We give a simple combinatorial formula for the twist, as well as an interpretation via knot theory. We prove several results about the twist, such as the fact that it is an integer and that it has additive properties for suitable decompositions of a region.
\end{abstract}

\section{Introduction}


Tiling problems have received a lot of attention in the second half of the twentieth century, two-dimensional domino and lozenge tilings in particular. For instance, Kasteleyn~\cite{Kasteleyn19611209}, Conway~\cite{conway1990tiling}, Thurston~\cite{thurston1990}, Elkies, Propp et al.~\cite{jockusch1998random,cohn1996local,elkies1992alternating}, Kenyon and Okounkov~\cite{kenyonokounkov2006dimers,kenyonokounkov2006planar} have come up with very interesting techniques, ranging from abstract algebra to probability. More relevant to the discussion in this paper are the problems of flip accessibility (e.g., \cite{saldanhatomei1995spaces}).    


Attempts to generalize some of these techniques to the three-dimensional case were made. The problem of counting domino tilings, even of contractible regions, is known to be computationally hard (see \cite{pak2013complexity}), but some asymptotic results, even for higher dimensions, date as far back as 1966 (see \cite{hammersley1966limit,ciucu1998improved,friedland2005theory}). In a different direction, some ``typically two-dimensional'' properties were carried over to specific families of three-dimensional regions (see \cite{randall2000random,linde2001rhombus,bodini2007tiling}). 

Others have considered difficulties with connectivity by local moves in dimension higher than two (see, e.g., \cite{randall2000random}). We propose an algebraic invariant that could help understand the structure of connected component by flips in dimension three.

\begin{figure}[ht]
	
	\centering
    \includegraphics[width=0.4\textwidth]{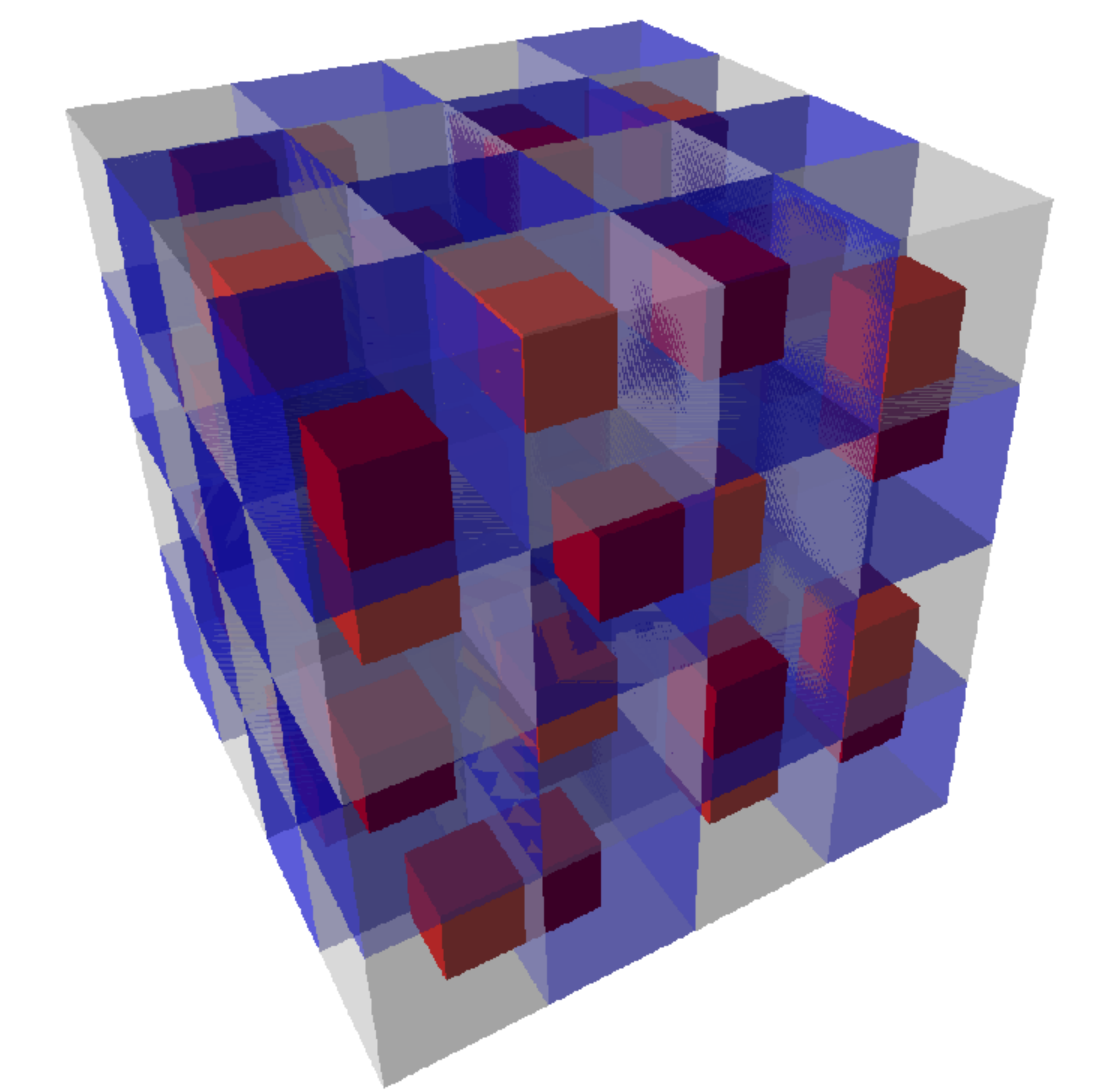}

		\caption{A tiling of a $4\times 4 \times 4$ box.}
			\label{fig:tiling444box}	

\end{figure}

In this paper, we investigate tilings of contractible regions by \emph{domino brick} pieces, or \emph{dominoes}, which are simply $2 \times 1 \times 1$ rectangular cuboids. An example of such a tiling is shown in Figure \ref{fig:tiling444box}. While this 3D representation of tilings may be attractive, it is also somewhat difficult to work with. Hence, we prefer to work with a 2D representation of tilings, which is shown in Figure \ref{fig:notation2Dexample}.

A key element in our study is the concept of a \emph{flip}, which is a straightforward generalization of the two-dimensional one. We perform a flip on a tiling by removing two (adjacent and parallel) domino bricks and placing them back in the only possible different position. The removed pieces form a $2 \times 2 \times 1$ slab, in one of three possible directions (see Figure \ref{fig:flipExample}). 
The \emph{flip connected component} of a tiling $t$
of a three-dimensional region $\cR$ is the set of all tilings of $\cR$
that can be reached from $t$ after a sequence of flips. 
It turns out that for large regions the number of flip connected components
is also large, but some of them may contain many tilings.
One of the aims of this paper is to study such connected components.

As in \cite{primeiroartigo}, we also consider the \emph{trit},
which is a move that happens within a $2 \times 2 \times 2$ cube with two opposite ``holes'', and which has an orientation (positive or negative). More precisely, we remove three dominoes, no two of them parallel, and place them back in the only other possible configuration (see Figure \ref{fig:negtrit_example}).   

\begin{figure}[hpt]
\centering
\def\svgwidth{0.8\columnwidth}
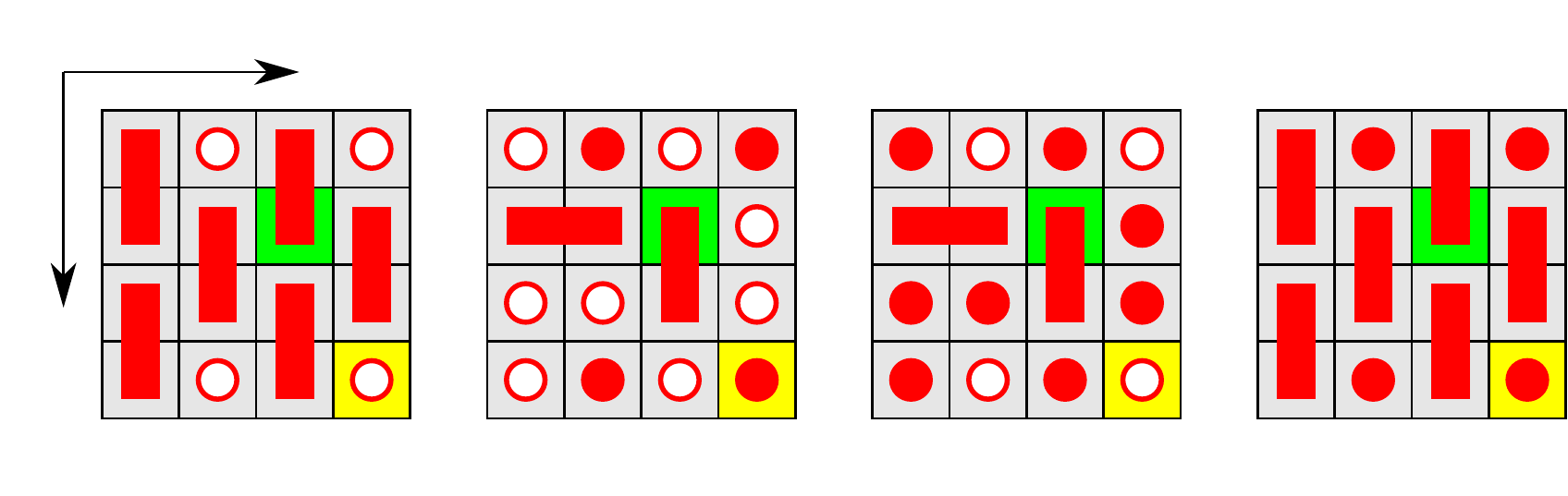
\caption{A tiling of the box $\cB = [0,4] \times [0,4] \times [0,4]$ box in our notation. The $x$ and $y$ axis are drawn, and $z$ points towards the paper, so that floors to the right have higher $z$ coordinates. Dominoes that are parallel to the $x$ or $y$ axis are represented as 2D dominoes, since they are contained in a single floor. Dominoes parallel to the $z$ axis are represented as circles, with the following convention: if the corresponding domino connects a floor with the floor to the left of it, the circle is painted red; otherwise, it is painted white. Thus, for example, in Figure \ref{fig:notation2Dexample}, each of the four white circles on the leftmost floor represents the same domino as the red circles on the floor directly to the right of it.
The squares highlighted in yellow represent cubes whose centers have the same $x$ and $y$ coordinates. Notice the top two yellow cubes are connected by a domino parallel to the $z$ axis, as well as the bottom two. The squares highlighted in green also represent cubes whose center have the same $x$ and $y$ coordinates, but the dominoes involving these cubes are not parallel to the $z$ axis.}
\label{fig:notation2Dexample}
\vspace{8pt}

\includegraphics[width=0.9\columnwidth]{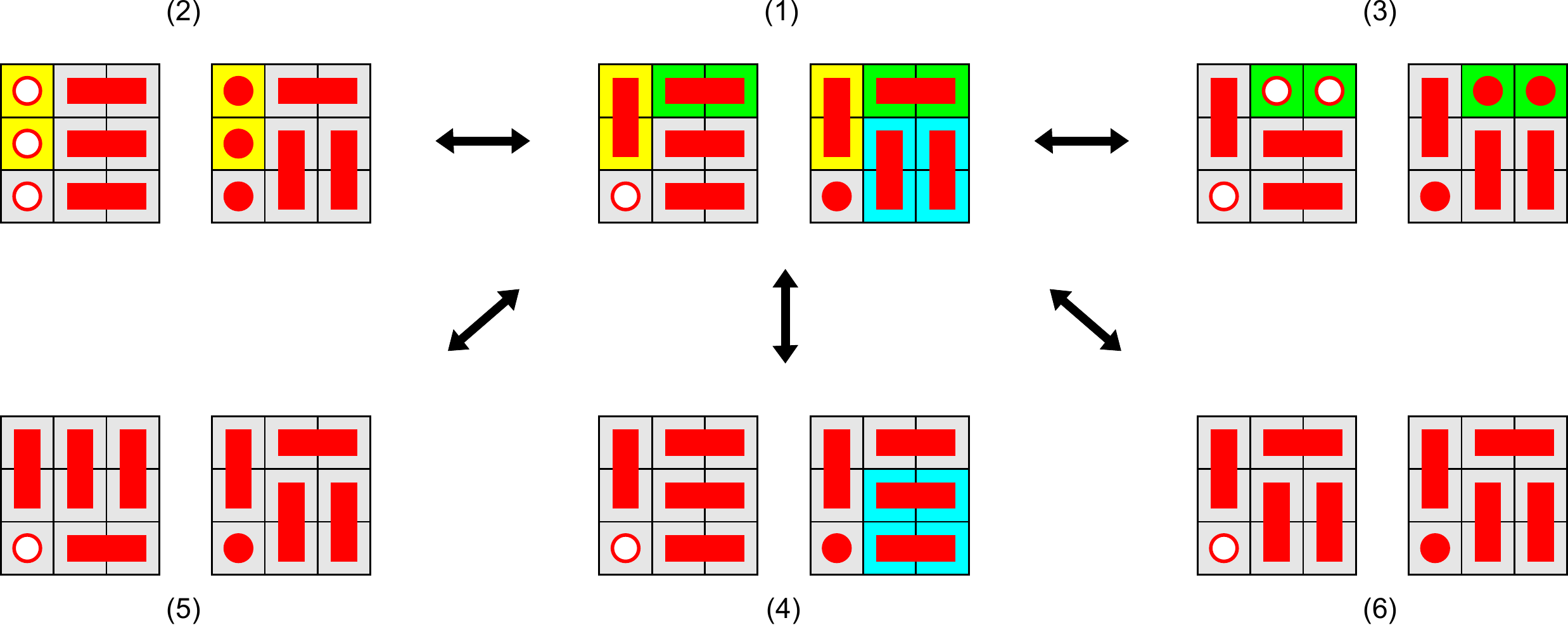}%
\caption{All flips available in tiling (1). The $2 \times 2 \times 1$ slabs involved in the flips taking (1) to (2), (3) and (4) are highlighted: they illustrate the three possible relative positions of dominoes in a flip.}%
\label{fig:flipExample}%

\vspace{10pt}

\includegraphics[width=0.6\columnwidth]{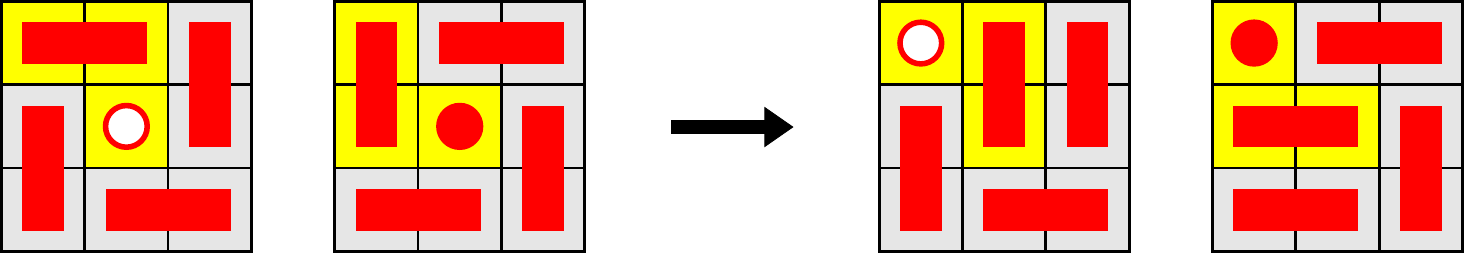}%
\caption{An example of a negative trit. The affected cubes are highlighted in yellow.}%
\label{fig:negtrit_example}%

\end{figure}

A \emph{cylinder} is a region of the form $\cD \times [0,n]$ (possibly rotated), where $\cD \subset \RR^2$ is a simply connected planar region with connected interior. In this paper, we introduce an algebraic invariant, the twist $\Tw(t)$, defined in Section \ref{sec:combTwistBoxes} for tilings of a cylinder.

In \cite{primeiroartigo}, we study cylinders with $n = 2$, called duplex regions. Although they are related to the general theory, tilings of these regions have some interesting characteristics of their own; in particular, we can define a polynomial $P_t(q)$ for tilings of duplex regions which is invariant by flips and which is finer than the twist. However, this construction breaks down when the duplex region is embedded in a region with more floors (see \cite{primeiroartigo} for details).

%
\begin{theo}
\label{theo:main}
Let $\cR$ be a cylinder, and $t$ a tiling of $\cR$. The twist $\Tw(t)$ is an integer with the following properties:
\begin{enumerate}[label=\upshape(\roman*)]
	\item \label{item:flipInvariance} If a tiling $t_1$ is reached from $t_0$ after a flip, then $\Tw(t_1) = \Tw(t_0)$.
	\item \label{item:tritDifference} If a tiling $t_1$ is reached from $t_0$ after a single positive trit, then $\Tw(t_1) - \Tw(t_0) = 1$.
	\item \label{item:duplexes} If $\cR$ is a duplex region, then $\Tw(t) = P_t'(1)$ for any tiling $t$ of $\cR$.   
	\item \label{item:multiplexUnion} Suppose a cylinder $\cR = \bigcup_{1 \leq i \leq m} \cR_i$, where each $\cR_i$ is a cylinder (they need not have the same axis) and such that $i \neq j \Rightarrow \interior(\cR_i) \cap \interior(\cR_j) \neq \emptyset$. Then there exists a constant $K \in \ZZ$ such that, for any family $(t_i)_{1 \leq i \leq n}$, $t_i$ a tiling of $\cR_i$,
	$$\Tw\left(\bigsqcup_{1 \leq i \leq m} t_{i}\right) = K + \sum_{1 \leq i \leq m} \Tw(t_{i}).$$
	\end{enumerate}
	\end{theo}

The definitions of twist are somewhat technical and involve a relatively lengthy discussion. We shall give two different but equivalent definitions: the first one, given in Section \ref{sec:combTwistBoxes}, is a sum over pairs of dominoes.  At first sight, this formula gives a number in $\frac{1}{4} \ZZ$ and depends on a choice of axis. However, it turns out that, for cylinders, this number is an integer, and different choices of axis yield the same result. The proof of this claim will be completed in Section \ref{sec:differentDirections}, and it relies on the second definition, which uses the concepts of writhe and linking number from knot theory (see, e.g., \cite{knotbook}).

One can ask whether the twist can be extended to a broader class of regions. The following (see \cite{tese}) holds:
let $\cR$ be a simply connected region (not necessarily a cylinder), $t_0$ and $t_1$ be two tilings of $\cR$. Suppose $\cB \supset \cR$ is a box and $t_*$ is a tiling of $\cB \setminus \cR$ (it is not true for arbitrary regions $\cR$ that $\cB$ and $t_*$ exist). Define $\TW(t_0,t_1) = \Tw(t_0 \sqcup t_*) - \Tw(t_1 \sqcup t_*)$: this turns out to depend neither on the choice of box $\cB$ nor on the choice of tiling $t_*$. Therefore, if we choose a base tiling $t_0$ and define $\Tw(t) = \TW(t,t_0)$, then $\Tw(t)$ satisfies items \ref{item:flipInvariance} and \ref{item:tritDifference} in Theorem \ref{theo:main}. Different choices of base tiling only alter the twist by an additive constant.

In addition to the combinatorial and knot-theoretic interpretations developed in this article, it is also possible to give homological interpretations for the twist. These homological constructions are reminiscent of the two-dimensional height functions (see~\cite{thurston1990}), although they behave more like ``height forms''. 
The concept of flux (or flow), as in \cite{saldanhatomei1995spaces} and \cite{saldanhatomeiAnnuli}, also becomes relevant.
Although we will not discuss these constructions here, this homological point of view inspired many of our definitions.

One might also ask what the possible values
for the twist of a certain region are.
Some results in this direction are proved in \cite{tese}: for instance,
it turns out that, for a box of dimensions $L\times M \times N$
with $L \ge M \ge N$ (and $LMN$ even), the maximum possible value
for the twist is of the order of $LMN^2$.

%

The present paper is structured in the following manner: 
Section \ref{sec:defsAndNotations} introduces some basic definitions and notations that will be used throughout the paper. 
In Section \ref{sec:combTwistBoxes}, we define the invariant for cylinders, and prove its most basic properties. 
%
In Sections \ref{sec:topologicalGroundwork}, \ref{sec:writheInterpretation}, \ref{sec:differentDirections} and \ref{sec:additiveProperties}, we present different aspects of a connection between the twist of tilings and a few classical concepts from knot theory.
Section \ref{sec:topologicalGroundwork} contains the ``topological groundwork'', which consists of a number of definitions and results that help establish topological interpretations of the twist, and which are extensively used in the sections that follow it.
In Section \ref{sec:writheInterpretation}, we introduce a different formula for the twist of cylinders, and show that this new formula allows us to prove (once again, via topology) that the twist must always be an integer.
In Section \ref{sec:differentDirections}, we prove that the value of the twist of cylinders does not depend on the choice of axis, which is one of the main results in the paper.
In Section \ref{sec:additiveProperties}, we discuss additive properties of the twist, and prove item \ref{item:multiplexUnion} in Theorem \ref{theo:main}.
Finally, Section \ref{sec:examples} contains some examples and counterexamples that help illustrate the theory. 
This paper closely corresponds to part of the first author's PhD thesis
\cite{tese};
the authors thank the examination board for helpful comments and suggestions.
The authors are also thankful for the generous support of
CNPq, CAPES and FAPERJ (Brazil). 

\section{Definitions and Notation}
\label{sec:defsAndNotations}

This section contains general notations and conventions that are used throughout the article, although definitions that involve a lengthy discussion or are intrinsic of a given section might be postponed to another section.  

If $n$ is an integer, $\plshalf{n}$ will denote $n + \half$ (in music theory, $D\sharp$ is a half tone higher than $D$ in pitch). We also define $\plshalf{\ZZ}$ to be the set $\{\plshalf{n} | n \in \ZZ\}$.

Given $\tv_1,\tv_2,\tv_3 \in \RR^3$, $\det(\tv_1,\tv_2,\tv_3) = \tv_1 \cdot (\tv_2 \times \tv_3)$ denotes the determinant of the $3 \times 3 \times 3$ matrix whose $i$-th line is $\tv_i$, $i=1,2,3$. If $\beta = (\vbeta_1, \vbeta_2, \vbeta_3)$ is a basis, write $\det(\beta) = \det(\vbeta_1, \vbeta_2, \vbeta_3)$.

We denote the three canonical basis vectors as $\ex = (1,0,0)$, $\ey = (0,1,0)$ and $\ez = (0,0,1).$ 
We denote by $\Delta = \{\ex, \ey, \ez\}$ the set of canonical basis vectors, and $\Phi = \{\pm \ex, \pm \ey, \pm \ez\}$. 
Let $\bB = \{\beta = (\vbeta_1,\vbeta_2,\vbeta_3) | \vbeta_i \in \Phi, \det(\beta) = 1\}$ be the set of positively oriented bases with vectors in $\Phi$. 

A \emph{basic cube} is a closed unit cube  in $\RR^3$ whose vertices lie in $\ZZ^3$. For $(x,y,z) \in \ZZ^3$, the notation $C\left(\plshalf{x} , \plshalf{y}, \plshalf{z}\right)$ denotes the basic cube $(x,y,z) + [0,1]^3$, i.e., the closed unit cube whose center is $\left(\plshalf{x},  \plshalf{y}, \plshalf{z}\right)$; it is \emph{white} (resp. \emph{black}) if $x + y + z$ is even (resp. odd). If $C = C\left(\plshalf{x} , \plshalf{y}, \plshalf{z}\right)$, define $\ccol(C) = (-1)^{x+y+z+1}$, or, in other words, $1$ if $C$ is black and $-1$ if $C$ is white. A \emph{region} is a finite union of basic cubes. A \emph{domino brick} or \emph{domino} is the union of two basic cubes that share a face. A \emph{tiling} of a region is a covering of this region by dominoes with pairwise disjoint interiors.


We sometimes need to refer to planar objects. Let $\pi$ denote either $\RR^2$ or a \emph{basic plane} contained in $\RR^3$, i.e., a plane with equation $x = k$, $y = k$ or $z = k$ for some $k \in \ZZ$. 
A \emph{basic square in} $\pi$ is a unit square $Q \subset \pi$ with vertices in $\ZZ^2$ (if $\pi = \RR^2$) or $\ZZ^3$.
 A \emph{planar region} $D \subset \pi$ is a finite union of basic squares. 

A region $\cR$ is a
\emph{cubiculated cylinder} or \emph{multiplex region} 
if there exist a basic plane $\pi$ with normal vector $\tv \in \Delta$,
a simply connected planar region $\cD \subset \pi$ with connected interior 
and a positive integer $n$ 
such that 
$$\cR = \cD + [0,n]\tv = \{p + s\tv | p \in \cD, s \in [0,n]\};$$ 
we usually call $\cR$ a \emph{cylinder} for brevity.
The cylinder $\cR$ above has \emph{base} $\cD$, \emph{axis} $\tv$ and \emph{depth} $n$. 
For instance, a cylinder with axis $\ez$ and depth $n$ can be written as $\cD \times [k,k+n]$, where $\cD \subset \RR^2$. 
A $\tv$-cylinder means a cylinder with axis $\tv$. A \emph{duplex region} (see \cite{primeiroartigo}) is a cylinder with depth $2$.

We sometimes want to point out that the hypothesis of simple connectivity (of a cylinder) is not being used: therefore, a \emph{pseudocylinder} with base $\cD$, axis $\tv$ and depth $n$ has the same definition as above, except that the planar region $\cD \subset \pi$ is only assumed to have connected interior (and is not necessarily simply connected).

A \emph{box} is a region of the form $\cB = [L_0, L_1] \times [M_0, M_1] \times [N_0, N_1]$, where $L_i, M_i, N_i \in \ZZ$. Boxes are special cylinders, in the sense that we can take any vector $\tv \in \Delta$ as the axis. In fact, boxes are the only regions that satisfy the definition of cylinder for more than one axis.

Regarding notation, Figures \ref{fig:notation2Dexample}, \ref{fig:flipExample} and \ref{fig:negtrit_example} were drawn with $\beta = (\ex,\ey,\ez)$ in mind. However, any $\beta \in \bB$ allows for such representations, as follows: we draw $\vbeta_3$ as perpendicular to the paper (pointing towards the paper). If $\pi = \vbeta_3^{\perp}$, we then draw each floor $\cR \cap (\pi + [n,n+1]\vbeta_3)$ as if it were a plane region. Floors are drawn from left to right, in increasing order of $n$.

The \emph{flip connected component} of a tiling $t$ of a region $\cR$ is the set of all tilings of $\cR$ that can be reached from $t$ after a sequence of flips. 



%
%

Suppose $t$ is a tiling of a region $\cR$, and let $\cB = [l,l+2] \times [m,m+2] \times [n,n+2]$, with $l,m,n \in \NN$. Suppose $\cB \cap \cR$ contains exactly three dominoes of $t$, no two of them parallel: notice that this intersection can contain six, seven or eight basic cubes of $\cR$. 
Also, a rotation (it can even be a rotation, say, in the $XY$ plane), can take us either to the left drawing or to the right drawing in Figure \ref{fig:postrit}.

\begin{figure}[ht]
\centering
\includegraphics[width=0.6\columnwidth]{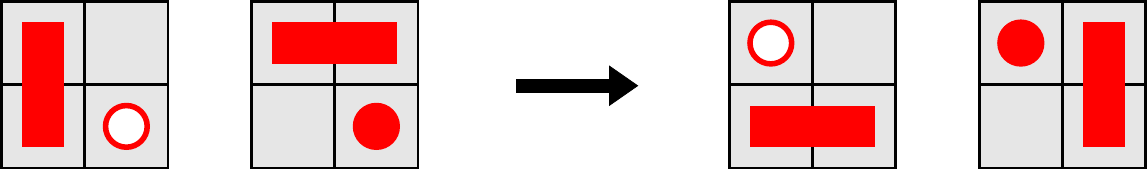}%
\caption{The anatomy of a positive trit (from left to right). The trit that takes the right drawing to the left one is a negative trit. The squares with no dominoes represent basic cubes that may or may not be in $\cR$ (see Figure \ref{fig:negtrit_example} for an example).}%
\label{fig:postrit}%
\end{figure}

If we remove the three dominoes of $t$ contained in $\cB \cap \cR$, there is only one other possible way we can place them back. This defines a move that takes $t$ to a different tiling $t'$ by only changing dominoes in $\cB \cap \cR$: this move is called a \emph{trit}. If the dominoes of $t$ contained in $\cB \cap \cR$ are a plane rotation of the left drawing in Figure \ref{fig:postrit}, then the trit is \emph{positive}; otherwise, it's \emph{negative}. Notice that the sign of the trit is unaffected by translations (colors of cubes don't matter) and rotations in $\RR^3$ (provided that these transformations take $\ZZ^3$ to $\ZZ^3$). A reflection, on the other hand, switches the sign (the drawing on the right can be obtained from the one on the left by a suitable reflection).

\section{The twist for cylinders}
\label{sec:combTwistBoxes}
For a domino $d$, define $\tv(d) \in \Phi$ to be the center of the black cube contained in $d$ minus the center of the white one. We sometimes draw $\tv(d)$ as an arrow pointing from the center of the white cube to the center of the black one.

For a set $X \subset \RR^3$ and $\vu \in \Phi$, we define the (\emph{open}) $\vu$-\emph{shade} of $X$ as
$$\cS^{\vu}(X) = \interior((X + [0,\infty)\vu) \setminus X) = \interior\left(\{ x + s \vu \in \RR^3 | x \in X, s \in [0,\infty)\} \setminus X\right),$$
where $\interior(Y)$ denotes the interior of $Y$.  
The \emph{closed} $\vu$-\emph{shade} $\clS^{\vu}(X)$ is the closure of $\cS^{\vu}(X)$.
We shall only refer to $\vu$-shades of unions of basic cubes or basic squares, such as dominoes.


\begin{figure}[ht]
\centering
\includegraphics[width=0.7\columnwidth]{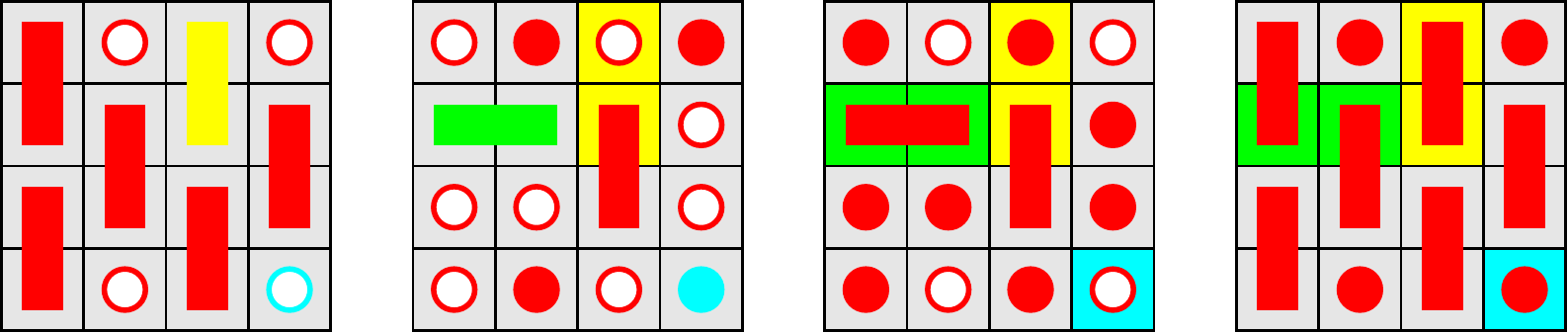}%
\caption{Tiling of a $4 \times 4 \times 4$ box, with three distinguished dominoes (painted yellow, green and cyan), whose $\ez$-shades are highlighted in the same color as they are. Notice that the yellow shade intersects four dominoes, the green shade intersects three, and the cyan shade, only one.}%
\label{fig:shadowExample}%
\end{figure} 

Given two dominoes $d_0$ and $d_1$ of $t$, we define the \emph{effect of} $d_0$ \emph{on} $d_1$ \emph{along} $\vu$, as:
$$\tau^{\vu}(d_0,d_1) = \begin{cases} \frac{1}{4} \det(\tv(d_1),\tv(d_0),\vu), &d_1 \cap \cS^{\vu}(d_0) \neq \emptyset \\ 0, &\mbox{otherwise} \end{cases}$$

In other words, $\tau^{\vu}(d_0, d_1)$ is zero unless the following three things happen: $d_1$ intersects the $\vu$-shade of $d_0$; neither $d_0$ nor $d_1$ are parallel to $\vu$; and $d_0$ is not parallel to $d_1$. When $\tau^{\vu}(d_0,d_1)$ is not zero, it's either $1/4$ or $-1/4$ depending on the orientations of $\tv(d_0)$ and $\tv(d_1)$.

For example, in Figure \ref{fig:shadowExample}, for $\vu = \ez$, the yellow domino $d_Y$ has no effect on any other domino: $\tau^{\ez}(d_Y,d) = 0$ for every domino $d$ in the tiling. The green domino $d_G$, however, affects the two dominoes in the rightmost floor which intersect its $\ez$-shade, and $\tau^{\ex}(d_G,d) = 1/4$ for both these dominoes.    

If $t$ is a tiling, we define the $\vu$-\emph{pretwist} as 
$$
T^{\vu}(t) = \sum_{d_0,d_1 \in t} \tau^{\vu}(d_0,d_1).
$$

For example, the tiling on the left of Figure \ref{fig:negtrit_example} has $\ez$-pretwist equal to $1$. To see this, notice that each of the four dominoes of the leftmost floor that are not parallel to $\ez$ has nonzero effect along $\ez$ on exactly one domino of the rightmost floor, and this effect is $1/4$ in each case. The reader may also check that the $\ez$-pretwist of the tiling in Figure \ref{fig:shadowExample} is $0$.


\begin{lemma}
\label{lemma:twistNegatingDirection}
For any pair of dominoes $d_0$ and $d_1$ and any $\vu \in \Phi$,
$\tau^{\vu}(d_0,d_1) = \tau^{-\vu}(d_1,d_0)$. In particular, for a tiling $t$ of a region we have $T^{-\vu}(t) = T^{\vu}(t)$. 
\end{lemma}
\begin{proof}
Just notice that
$d_1 \cap \cS^{\vu}(d_0) \neq\emptyset$ if and only if $d_0 \cap \cS^{-\vu}(d_1) \neq \emptyset,$ 
and $\det(\tv(d_1),\tv(d_0),\vu) = \det(\tv(d_0),\tv(d_1),-\vu)$.
\end{proof}
Translating both dominoes by a vector with integer coordinates clearly does not affect $\tau^{\vu}(d_0,d_1)$, as $\det(\tv(d_1),\tv(d_0),\vu) = \det(-\tv(d_1),-\tv(d_0),\vu)$. Therefore, if $t$ is a tiling and $f(p) = p + b$, where $b \in \ZZ^3$, then $T^{\vu}(f(t)) = T^{\vu}(t)$.

\begin{lemma}
\label{lemma:twistReflection}
Let $\cR$ be a region, and let $\vw \in \Delta$. Consider the reflection $r = r_{\vw}: \RR^3 \to \RR^3: p \mapsto p - 2(p\cdot \vw)\vw$; notice that $r(\cR)$ is a region.
If $t$ is a tiling of $\cR$ and $\vu \in \Phi$, then the tiling $r(t) = \{r(d), d \in t\}$ of $r(\cR)$ satisfies $T^{\vu}(r(t)) = - T^{\vu}(t)$. 
\end{lemma}
\begin{proof}
Given a domino $d$ of $t$, notice that $\tv(r(d)) = - r(\tv(d))$ and that $\cS^{\vu}(r(d)) = r(\cS^{r(\vu)}(d))$. Therefore, $r(d_1) \cap \cS^{\vu}(r(d_0)) \neq \emptyset \Leftrightarrow d_1 \cap S^{r(\vu)}(d_0) \neq \emptyset$ and
\begin{align*}
&\det(\tv(r(d_1)),\tv(r(d_0)),\vu) = \det(-r(\tv(d_1)), -r(\tv(d_0)), \vu)\\ 
&= \det(r(\tv(d_1)), r(\tv(d_0)), r(r(\vu))) = - \det(\tv(d_1), \tv(d_0), r(\vu)).  
\end{align*}
Therefore, $\tau^{\vu}(r(d_0),r(d_1)) = - \tau^{r(\vu)}(d_0,d_1)$ and thus $T^{\vu}(r(t)) = - T^{r(\vu)}(t)$. Since $r(\vu) = \pm \vu$, Lemma \ref{lemma:twistNegatingDirection} implies that $T^{\vu}(r(t)) = - T^{\vu}(t)$, completing the proof.
\end{proof}

 A natural question at this point concerns how the choice of $\vu$ affects $T^{\vu}$. It turns out that it will take us some preparation before we can tackle this question.

\begin{prop}
\label{prop:equalTwistsMultiplex}
If $\cR$ is a cylinder and $t$ is a tiling of $\cR$,
$$T^{\ex}(t) = T^{\ey}(t) = T^{\ez}(t) \in \ZZ.$$
\end{prop}
\begin{proof}
Follows directly from Propositions \ref{prop:equalTwistsIntNEven} and \ref{prop:equalTwistsIntNOdd} below.
\end{proof}
This result doesn't hold in pseudocylinders or in more general simply connected regions; see Section \ref{sec:examples} for counterexamples.




\begin{defin}
\label{def:twist}
For a tiling $t$ of a cylinder $\cR$, we define the \emph{twist} $\Tw(t)$ as
\begin{equation*}
\Tw(t) = T^{\ex}(t) = T^{\ey}(t) = T^{\ez}(t).
\end{equation*} 
\end{defin}

Until Section \ref{sec:differentDirections}, we will not use Proposition \ref{prop:equalTwistsMultiplex}, and will only refer to pretwists.

Let $\vu \in \Delta$, and let $\beta = (\vbeta_1,\vbeta_2,\vbeta_3) \in \bB$ be such that $\vbeta_3 = \vu$.
A region $\cR$ is said to be \emph{fully balanced with respect to} $\vu$ if for each square $Q = p + [0,2]\vbeta_1 + [0,2] \vbeta_2$, where $p \in \ZZ^3$ and $Q \subset \cR$, each of the two sets $\cA^{\vu} = \cR \cap \clS^{\vu}(Q)$ and $\cA^{-\vu} = \cR \cap \clS^{-\vu}(Q)$  contains as many black cubes as white ones. In other words,
$$ \sum_{C \subset \cA^{\vu}} \ccol(C) =  \sum_{C \subset \cA^{-\vu}} \ccol(C) = 0.$$ 
$\cR$ is \emph{fully balanced} if it is fully balanced with respect to each $\vu \in \Delta$.

\begin{lemma}
\label{lemma:fullyBalancedMultiplex}
Every pseudocylinder (in particular, every cylinder) is fully balanced.
\end{lemma}
\begin{proof}
Let $\cR$ be a pseudocylinder with base $\cD$ and depth $n$, let $\vu \in \Delta$ and let $Q = p_0 + [0,2]\vbeta_1 + [0,2] \vbeta_2 \subset \cR$, where $\beta \in \bB$ is such that $\vbeta_3 = \vu$ and $p_0 \in \ZZ^3$.
Consider $\cA^{\pm\vu} = \cR \cap \clS^{\pm\vu}(Q)$.

If $\vu$ is the axis of the pseudocylinder, then $Q = Q' + k\vu$, for some square $Q' \subset \cD$ and some $0 \leq k \leq n$. Now $\cA^{\vu} = Q' + [k,n]\vu$, which clearly contains $2(n-k)$ black cubes and $2(n-k)$ white ones; similarly, $\cA^{-\vu} = Q' + [0,k]\vu$ contains $2k$ black cubes and $2k$ white ones.

If $\vu$ is perpendicular to the axis of the pseudocylinder, assume without loss of generality that $\vbeta_1$ is the axis. Let $\Pi$ denote the orthogonal projection on $\cD$, and let $\cD^{\pm} = \clS^{\pm\vu}(\Pi(Q)) \cap \cD$, which are planar regions, since they are unions of squares of $\cD$. If $p_0 - \Pi(p_0) = k \vbeta_1$, we have $\cA^{\pm\vu} = \cD^{\pm} + [k,k+2]\vbeta_1$, which clearly has the same number of black squares as white ones. 
%
%
%
%
\end{proof}



\begin{prop}
\label{prop:flipsAndTrits}
Let $\cR$ be a region that is fully balanced with respect to $\vu \in \Phi$.
\begin{enumerate}[label=\upshape(\roman*),topsep = 0.1px]
	\item \label{item:flips} If a tiling $t_1$ of $\cR$ is reached from $t_0$ after a flip, then $T^{\vu}(t_0) = T^{\vu}(t_1)$
	\item \label{item:trits} If a tiling $t_1$ of $\cR$ is reached from $t_0$ after a single positive trit, then $T^{\vu}(t_1) = T^{\vu}(t_0) + 1$.
\end{enumerate}
\end{prop}
\begin{proof}
In this proof, $\vu$ points towards the paper in all the drawings.
We begin by proving \ref{item:flips}.
Suppose a flip takes the dominoes $d_0$ and $\tilde{d_0}$ in $t_0$ to $d_1$ and $\tilde{d_1}$ in $t_1$. 
Notice that $\tv(d_0) = -\tv(\tilde{d_0})$ and $\tv(d_1) = - \tv(\tilde{d_1})$. 
For each domino $d \in t_0 \cap t_1$, define 
$$E^{\pm\vu}(d) = \tau^{\pm\vu}(d,d_1) + \tau^{\pm\vu}(d,\tilde{d_1}) - \tau^{\pm\vu}(d,d_0) - \tau^{\pm\vu}(d,\tilde{d_0}).$$ 
Notice that
$$T^{\vu}(t_1) - T^{\vu}(t_0) = \sum_{d \in t_0 \cap t_1} {E^{\vu}(d) + E^{-\vu}(d)}.$$


\begin{case}
\label{case:d1parz}
Either $d_0$ or $d_1$ is parallel to $\vu$.
\end{case}
\begin{figure}[ht]%
\centering
\includegraphics[width=\columnwidth]{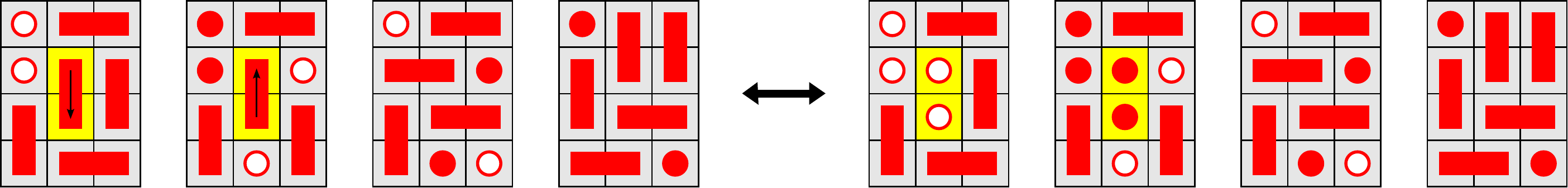}%
\caption{An example of Case \ref{case:d1parz}, where the black arrows represent $\tv(d_0)$ and $\tv(\tilde{d_0})$. It is clear that the effects of $d_0$ and $\tilde{d_0}$ cancel out.}%
\label{fig:flip_shadow_c1}%
\end{figure}
Assume, without loss of generality, that $d_1$ (and thus also $\tilde{d_1}$) is parallel to $\vu$.
By definition, $\tau^{\pm\vu}(d,d_1) = \tau^{\pm\vu}(d,\tilde{d_1}) = 0$ for each domino $d$. Now notice that $d_0$ and $\tilde{d_0}$ are parallel and in adjacent floors (see Figure \ref{fig:flip_shadow_c1}) : since $\tv(d_0) = -\tv(\tilde{d_0})$, it follows that $\tau^{\pm\vu}(d,d_0) + \tau^{\pm\vu}(d,\tilde{d_1}) = 0$ for each domino $d$, so that $E^{\pm\vu}(d) = 0$ and thus $T^{\vu}(t_1) = T^{\vu}(t_0)$. 
%
%
%

\begin{case}
\label{case:d1notparz}
Neither $d_0$ nor $d_1$ is parallel to $\vu$.
\end{case}
\begin{figure}[ht]%
\centering
\subfloat[The flip position is highlighted in yellow in both tilings, and $\cA^{\vu}$ is highlighted in green. The vectors $\tv(d)$ have been drawn for the most relevant dominoes.]{\includegraphics[width=0.4\columnwidth]{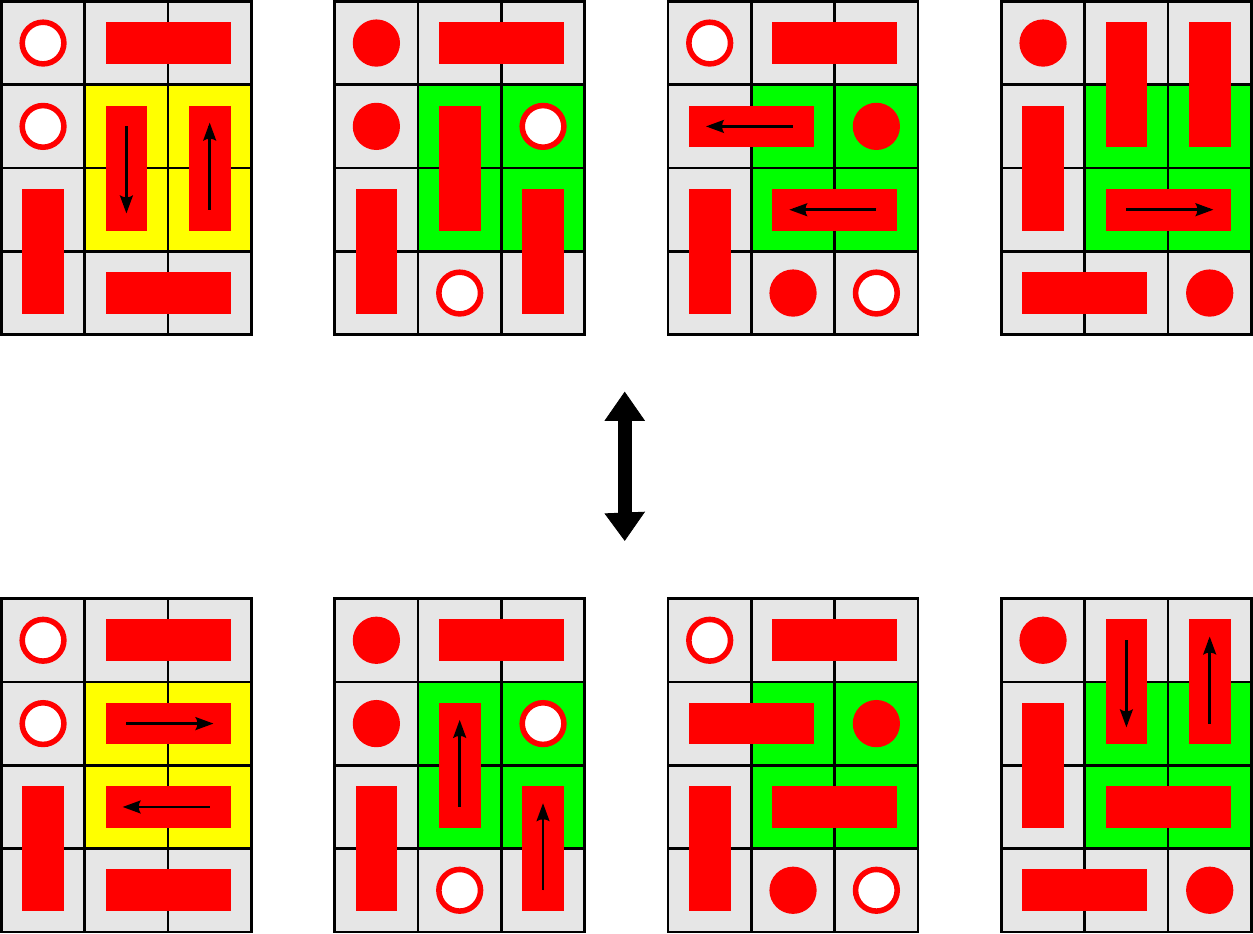}\label{fig:flip_shadow_c2}} \qquad \qquad
\subfloat[This refers to the tilings in \protect\subref{fig:flip_shadow_c2}, but only the arrows are drawn (not the dominoes). Notice that we have drawn $-\tv(d_0)$ and $-\tv(\tilde{d_0})$.]{\def\svgwidth{0.4\columnwidth}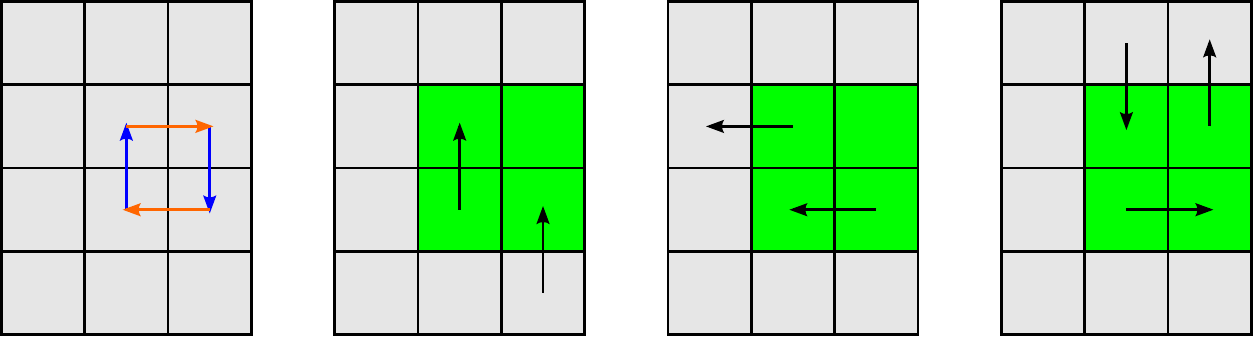\label{fig:flip_shadow_c2_scheme}}
\caption{Example of a flip in Case \ref{case:d1notparz}, together with a schematic drawing portraying $\tv(d)$ for the relevant dominoes.}
\label{fig:flip_shadow_c2_complete}
\end{figure}

In this case, $d_0 \cup \tilde{d_0} = d_1 \cup \tilde{d_1} = Q + [0,1]\vu \subset \cR$ for some square $Q$ of side $2$ and normal vector $\vu$. 

Notice that $\clS^{\vu}(d_0) \cup \clS^{\vu}(\tilde{d_0}) = \clS^{\vu}(d_1) \cup \clS^{\vu}(\tilde{d_1}) = \clS^{\vu}(Q + \vu)$; let $\cA^{\vu} = \cR \cap \clS^{\vu}(Q + \vu)$. 

Let $d$ be a domino that is completely contained in $\Aout$: we claim that $\tau(d_0,d) + \tau(\tilde{d_0},d) = 0 = \tau(d_1,d) + \tau(\tilde{d_1},d)$. This is obvious if $d$ is parallel to $\vu$; if not, we can switch the roles of $t_0$ and $t_1$ if necessary and assume that $d$ is parallel to $d_0$, which implies that $\tau(d_0,d) = \tau(\tilde{d_0},d) = 0$. Now notice that $d$ is in the $\vu$-shades of both $d_1$ and $\tilde{d_1}$, so that $\tau(d_1,d) = - \tau(\tilde{d_1},d)$. Hence, if $d \subset \cA^{\vu}$ (or if $d \cap \cA^{\vu} = \emptyset$), $E^{-\vu}(d) = 0$.

For dominoes $d$ that intersect $\Aout$ but are not contained in it, first observe that by switching the roles of $t_0$ and $t_1$ and switching the colors of the cubes (i.e., translating) if necessary, we may assume that the vectors are as shown in Figure \ref{fig:flip_shadow_c2}. By looking at Figure \ref{fig:flip_shadow_c2_scheme} and working out the possible cases, we see that
$$
E^{-\vu}(d) = 
\begin{cases}
 -\frac{1}{4}, &\text{if } \tv(d) \text{ points into } \Aout; \\
\frac{1}{4}, &\text{if } \tv(d) \text{ points away from } \Aout.\end{cases}
 $$  

Now for such dominoes, $\tv(d)$ points away from the region if and only if $d$ intersects a white cube of $\Aout$, and points into the region if and only if $d$ intersects a black cube in $\Aout$: hence,
$$\sum_{d \in t_0 \cap t_1} E^{-\vu}(d) = \sum_{C \subset \cA^{\vu}} (-\ccol(C)) = 0,$$ 
because $\cR$ is fully balanced with respect to $\vu$.
A completely symmetrical argument shows that $\sum_{d \in t_0 \cap t_1} E^{\vu}(d) = 0$, so we are done.

We now prove \ref{item:trits}. Suppose $t_1$ is reached from $t_0$ after a single positive trit. By rotating $t_0$ and $t_1$ in the plane $\vu^{\perp} = \{ \vw | \vw \cdot \vu = 0\}$ (notice that this does not change $T^{\vu}$), we may assume without loss of generality that the dominoes involved in the positive trit are as shown in Figure \ref{fig:postrit}. Moreover, by translating if necessary, we may assume that the vectors $\tv(d)$ are as shown in Figure \ref{fig:postrit_diff}.

A trit involves three dominoes, no two of them parallel. Since dominoes parallel to $\vu$ have no effect along $\vu$, we consider only the four dominoes involved in the trit that are not parallel to $\vu$: $d_0, \tilde{d_0} \in t_0$, and $d_1, \tilde{d_1} \in t_1$. Define $E^{\pm\vu}$ with the same formulas as before.  

By looking at Figure \ref{fig:postrit}, the reader will see that $\tau(d_0,\tilde{d_0}) + \tau(\tilde{d_0},d_0) = -1/4$ and $\tau(d_1,\tilde{d_1}) + \tau(\tilde{d_1},d_1) = 1/4$. 

Let $D = d_0 \cup \tilde{d_0} \cup d_1 \cup \tilde{d_1}$: $\clS^{\vu}(D)$ is shown in Figure \ref{fig:postrit_diff}. 
$D$ contains a single square $Q$ of side $2$ and normal vector $\vu$.
Define $\Aout = \clS^{\vu}(D) \cap \cR$, and notice that (see Figure \ref{fig:postrit_diff}) $\clS^{\vu}(Q) \cap \cR = \Aout \cup C_1 \cup C_2 \cup C_3$, where $C_i$ are three basic cubes: if we look at the arrows in Figure \ref{fig:postrit_diff}, we see that two of them are white and one is black. Since $\cR$ is fully balanced with respect to $\vu$, 
$$\sum_{C \subset \Aout} \ccol(C) = \sum_{C \subset \clS^{\vu}(Q) \cap \cR}{\ccol(C)} - \sum_{1 \leq i \leq 3}\ccol(C_i) = 1.$$  



\begin{figure}[ht]%
\centering
\def\svgwidth{0.5\columnwidth}%
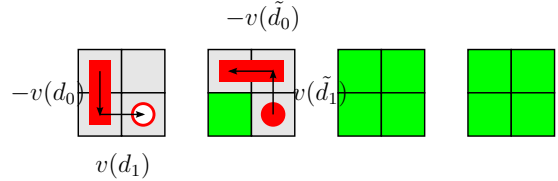
\caption{Illustration of a positive trit position: the portrayed dominoes belong to $t_0$, and the green cubes represent $S^{\vu}(D)$. The vectors $-\tv(d_0)$, $-\tv(\tilde{d_0})$, $\tv(d_1)$ and $\tv(\tilde{d_1})$ are shown.}%
\label{fig:postrit_diff}%
\end{figure} 

By looking at Figure \ref{fig:postrit_diff}, we see that we have a situation that is very similar to Figure \ref{fig:flip_shadow_c2_scheme}; for each $d \in t_0 \cap t_1$, we have 
$$E^{-\vu}(d) = \begin{cases} 0, &\text{if } d \subset \Aout \text{ or } d \cap \Aout = \emptyset;\\
 \frac{1}{4}, &\text{if } \tv(d) \text{ points into } \Aout; \\
-\frac{1}{4}, &\text{if } \tv(d) \text{ points away from } \Aout\end{cases} $$  
(when we say that $\tv(d)$ points into or away from $\Aout$, we are assuming that $d$ intersects one cube of $\Aout$).
Hence, $$\sum_{d \in t_0 \cap t_1}E^{-\vu}(d) = \frac{1}{4} \sum_{C \subset \cA^{\vu}}\ccol(C) = \frac{1}{4}.$$

A completely symmetrical argument shows that $\sum_{d \in t_0 \cap t_1}E^{\vu}(d) = 1/4$, and hence
\begin{align*}
&T^{\vu}(t_1) - T^{\vu}(t_0) = (\tau(d_1,\tilde{d_1}) + \tau(\tilde{d_1},d_1)) - (\tau(d_0,\tilde{d_0}) + \tau(\tilde{d_0},d_0))\\ &+ \sum_{d \in t_0 \cap t_1}E^{-\vu}(d) + \sum_{d \in t_0 \cap t_1}E^{\vu}(d)
= \frac{1}{4} + \frac{1}{4} + \frac{1}{4} + \frac{1}{4}  = 1,
\end{align*}
which completes the proof.   
\end{proof}
\section{Topological groundwork for the twist}
\label{sec:topologicalGroundwork}

In this section, we develop a topological interpretation of tilings and twists. Dominoes are (temporarily) replaced by dimers, which, although formally different objects, are really just a different way of looking at dominoes. Although we will tend to work with dimers in this and the following section, we may in later sections switch back and forth between these two viewpoints.

Let $\cR$ be a region. A \emph{segment} $\ell$ of $\cR$ is a straight line of unit length connecting the centers of two cubes of $\cR$; in other words, $\ell:[0,1] \to \RR^3$ with $\ell(s) = p_0 + (p_1 - p_0)s$, where $p_0$ and $p_1$ are the centers of two cubes that share a face: this segment is a \emph{dimer} if $p_0 = \ell(0)$ is the center of a white cube. We define $\tv(\ell) = \ell(1) - \ell(0)$ (compare this with the definition of $\tv(d)$ for a domino $d$). If $\ell$ is a segment, $(-\ell)$ denotes the segment $s \mapsto \ell(1-s)$: notice that either $\ell$ or $-\ell$ is a dimer.

Two segments $\ell_0$ and $\ell_1$ are \emph{adjacent} if $\ell_0 \cap \ell_1 \neq \emptyset$ (here we make the usual abuse of notation of identifying a curve with its image in $\RR^3$); nonadjacent segments are \emph{disjoint}. In particular, a segment is always adjacent to itself. 

A \emph{tiling} of $\cR$ by dimers is a set of pairwise disjoint dimers such that the center of each cube of $\cR$ belongs to exactly one dimer of $t$. If $t$ is a tiling, $(-t)$ denotes the set of segments $\{-\ell | \ell \in t\}$.   

Given a map $\gamma:[m,n] \to \RR^3$, a segment $\ell$ and an integer $k \in [m,n-1]$, we abuse notation by making the identification $\gamma|_{[k,k+1]} = \ell$ if $\gamma(s) = \ell(s-k)$ for each $s \in [k,k+1]$. 
A \emph{curve} of $\cR$ is a map $\gamma:[0,n] \to \RR^3$ such that $\gamma|_{[k,k+1]}$ is (identified with) a segment of $\cR$ for $k=0,1,\ldots,n-1$.
We make yet another abuse of notation by also thinking of $\gamma$ as a sequence or set of segments of $\cR$, and we shall write $\ell \in \gamma$ to denote that $\ell = \gamma|_{[k,k+1]}$ for some $k$. 

A curve $\gamma:[0,n] \to \RR^3$ of $\cR$ is \emph{closed} if $\gamma(0) = \gamma(n)$; it is \emph{simple} if $\gamma$ is injective in $[0,n)$. A closed curve $\gamma:[0,2] \to \RR^3$ of $\cR$ is called \emph{trivial}: notice that, in this case, $\gamma|_{[0,1]} = -(\gamma|_{[1,2]})$ (when identified with their respective segments of $\cR$). A \emph{discrete rotation} on $[0,n]$ is a function $\rho:[0,n] \to [0,n]$ with $\rho(s) = (s + k) \bmod n$, for a fixed $k \in \ZZ$. If $\gamma_0:[0,n] \to \RR^3$ and $\gamma_1:[0,m] \to \RR^3$ are two closed curves, we say $\gamma_0 = \gamma_1$ if $n = m$ and $\gamma_1 = \gamma_0 \circ \rho$ for some discrete rotation $\rho$ on $[0,n]$.

Given two tilings $t_0$ and $t_1$, there exists a unique (up to discrete rotations) finite set of disjoint closed curves $\Gamma(t_0,t_1) = \{\gamma_i | 1 \leq i \leq m\}$ such that $t_0 \cup (-t_1) = \{ \ell | \ell \in \gamma_i \text{ for some } i\}$ and such that every nontrivial $\gamma_i$ is simple. Figure \ref{fig:auxiliaryLinesExample} shows an example. We define $\Gamma^*(t_0,t_1) := \{\gamma \in \Gamma(t_0,t_1) | \gamma \text{ nontrivial } \}$. 

Translating effects from the world of dominoes to the world of dimers is relatively straightforward. For $\vu \in \Phi$, $\Pi^{\vu}$ will denote the orthogonal projection on the plane $\pi^{\vu} = \vu^{\perp} = \{\vw \in \RR^3 | \vw \cdot \vu = 0\}$. Given two segments $\ell_0$ and $\ell_1$, we set:

$$
\tau^{\vu}(\ell_0, \ell_1) = 
\begin{cases}
\frac{1}{4} \det(\tv(\ell_1), \tv(\ell_0), \vu), &\Pi^{\vu}(\ell_0) \cap \Pi^{\vu}(\ell_1) \neq \emptyset, \ell_0(0) \cdot \vu < \ell_1(0) \cdot \vu; \\
0, &\text{otherwise.} 
\end{cases}
$$
Notice that this definition is analogous to the one given in Section \ref{sec:combTwistBoxes} for dominoes.


The definition of $\tau^{\vu}$ is given in terms of the orthogonal projection $\Pi^{\vu}$. From a topological viewpoint, however, this projection is not ideal, because it gives rise to nontransversal intersections between projections of segments. In order to solve this problem, we consider small perturbations of these projections.


Recall that $\bB$ is the set of positively oriented basis $\beta = (\vbeta_1,\vbeta_2,\vbeta_3)$ with vectors in $\Phi$. If $\beta \in \bB$ and $a,b \in \RR$, $\Pi^{\beta}_{a,b}$ will be used to denote the projection on the plane $\pi^{\vbeta_3} = \vbeta_3^{\perp} = \{\vu \in \RR^3 | \vu \cdot \vbeta_3 = 0\}$ whose kernel is the subspace (line) generated by the vector $\vbeta_3 + a \vbeta_1 + b \vbeta_2$. For instance, if $\beta = (\ex, \ey, \ez)$ is the canonical basis, $\Pi^{\beta}_{a,b}(x,y,z) = (x - az, y - bz, 0)$. 
Notice that $\Pi^{\beta}_{0,0} = \Pi^{\vbeta_3}$ is the orthogonal projection on the plane $\pi^{\vbeta_3}$, and, for small $(a,b) \neq (0,0)$, $\Pi^{\beta}_{a,b}$ is a nonorthogonal projection on $\pi^{\vbeta_3}$ which is a slight perturbation of $\Pi^{\vbeta_3}$.

Given $\beta \in \bB$, $\vu = \vbeta_3$ and small nonzero $a,b \in \RR$, set the \emph{slanted effect}
$$
\tau^{\beta}_{a,b}(\ell_0,\ell_1) =
\begin{cases}
\det(\tv(\ell_1), \tv(\ell_0), \vu), &\Pi^{\beta}_{a,b}(\ell_0) \cap \Pi^{\beta}_{a,b}(\ell_1) \neq \emptyset, \vu \cdot \ell_0(0)  < \vu \cdot \ell_1(0);\\
0, &\text{otherwise.} 
\end{cases}
$$ 

Recall from knot theory the concept of crossing (see, e.g., \cite[p.18]{knotbook}). Namely, if $\gamma_0: I_0 \to \RR^3$, $\gamma_1: I_1 \to \RR^3$ are two continuous curves, $s_j \in \interior(I_j)$ and $\Pi$ is a projection from $\RR^3$ to a plane, then $(\Pi, \gamma_0, s_0, \gamma_1, s_1)$ is a \emph{crossing} if $\gamma_0(s_0) \neq \gamma_1(s_1)$ but $\Pi(\gamma_0(s_0)) = \Pi(\gamma_1(s_1))$. If, furthermore, $\gamma_j$ is of class $C^1$ in $s_j$ and the vectors $\gamma_1'(s_1)$, $\gamma_0'(s_0)$ and $\gamma_1(s_1) - \gamma_0(s_0)$ are linearly independent, then the crossing is \emph{transversal}; its \emph{sign} is the sign of $\det(\gamma_1'(s_1), \gamma_0'(s_0),\gamma_1(s_1) - \gamma_0(s_0))$. We are particularly interested in the case where the curves are segments of a region $\cR$.

%

For a region $\cR$ and $\vu \in \Phi$, we define the $\vu$-\emph{length} of $\cR$ as
 $$N = \max_{p_0, p_1 \in R} |\vu \cdot (p_0 - p_1)|.$$

\begin{lemma}
\label{lemma:transversalCrossings}
Let $\cR$ be a region, and fix $\beta \in \bB$. Let $N$ be the $\vbeta_3$-length of $\cR$, and let $a,b \in \RR$ with $0 < |a|,|b| < 1/N$.
Then $\tau^{\beta}_{a,b}(\ell_0,\ell_1) + \tau^{\beta}_{a,b}(\ell_1,\ell_0) \neq 0$ if and only if there exist $s_0, s_1 \in [0,1]$ such that $\ell_0(s_0) \neq \ell_1(s_1)$ but $\Pi^{\beta}_{a,b}(\ell_0(s_0)) = \Pi^{\beta}_{a,b}(\ell_1(s_1))$. 

Moreover, if the latter condition holds for $s_0, s_1$, then $(\Pi^{\beta}_{a,b},\ell_0, s_0, \ell_1, s_1)$ is a transversal crossing whose sign is given by $\tau^{\beta}_{a,b}(\ell_0,\ell_1) + \tau^{\beta}_{a,b}(\ell_1,\ell_0)$.

%
%
 %
\end{lemma}
\begin{proof}
Suppose $\tau^{\beta}_{a,b}(\ell_0,\ell_1) + \tau^{\beta}_{a,b}(\ell_1,\ell_0) \neq 0$. We may without loss of generality assume $\tau^{\beta}_{a,b}(\ell_0,\ell_1) \neq 0$. By definition, we have $\Pi^{\beta}_{a,b}(\ell_0(s_0)) = \Pi^{\beta}_{a,b}(\ell_1(s_1))$ for some $s_0, s_1 \in [0,1]$ and $\vbeta_3 \cdot \ell_0(0) < \vbeta_3 \cdot \ell_1(0)$. Since $\det(\tv(\ell_1), \tv(\ell_0), \vbeta_3) \neq 0$, we have
$$\vbeta_3 \cdot \ell_0(s_0) = \vbeta_3 \cdot (\ell_0(0) + s_0 \tv(\ell_0)) = \vbeta_3 \cdot \ell_0(0) <\vbeta_3 \cdot \ell_1(0) = \vbeta_3 \cdot \ell_1(s_1),$$
and thus $\ell_0(s_0) \neq \ell_1(s_1)$. 

Conversely, suppose $\ell_0(s_0) \neq \ell_1(s_1)$ but $\Pi^{\beta}_{a,b}(\ell_0(s_0)) = \Pi^{\beta}_{a,b}(\ell_1(s_1))$: this can be rephrased as
\begin{equation}
\ell_1(s_1) - \ell_0(s_0) = c (\vbeta_3 + a \vbeta_1 + b \vbeta_2)
\label{eq:projectionKernel}
\end{equation}
for some $c \neq 0$. Notice that $c = \vbeta_3 \cdot (\ell_1(s_1) - \ell_0(s_0))$, so that $|c| \leq N$.

We now observe that $\det(\tv(\ell_1), \tv(\ell_0), \vbeta_3) \neq 0$. Suppose, by contradiction, that $\det(\tv(\ell_1), \tv(\ell_0), \vbeta_3) = 0$. Then, at least one of the following statements must be true: $\vbeta_1 \cdot \tv(\ell_0) = \vbeta_1 \cdot \tv(\ell_1) = 0$; or $\vbeta_2 \cdot \tv(\ell_0) = \vbeta_2 \cdot \tv(\ell_1) = 0$. Assume that the first statement holds (i.e., $\vbeta_1 \cdot \tv(\ell_i) = 0$). 
By definition of segment, 
$\ell_i(s_i) = \ell_i(0) + s_i \tv(\ell_i)$. By taking the inner product with $\vbeta_1$ on both sides of \eqref{eq:projectionKernel}, $a c = \vbeta_1 \cdot (\ell_1(s_1) - \ell_0(s_0)) = \vbeta_1 \cdot (\ell_1(0) - \ell_0(0))$. Now $\ell_0(0), \ell_1(0) \in (\plshalf{\ZZ})^3$, so that $a c = \vbeta_1 \cdot (\ell_1(0) - \ell_0(0)) \in \ZZ$. Since $|a| < 1/N$, $|ac| < 1$ and thus $c = 0$, which is a contradiction.  

Finally, since  $\vbeta_3 \cdot \tv(\ell_0) = \vbeta_3 \cdot \tv(\ell_1) = 0$, we have $\vbeta_3 \cdot (\ell_1(0) - \ell_0(0)) = \vbeta_3 \cdot (\ell_1(s_0) - \ell_0(s_1)) = c \neq 0$. From the definition of $\tau^{\beta}_{a,b}$, we see that $\tau^{\beta}_{a,b}(\ell_0, \ell_1) + \tau^{\beta}_{a,b}(\ell_1, \ell_0) \neq 0$. 

To see the last claim, we first note that $s_i \in (0,1)$: since $\tv(\ell_i) \in \{\pm\vbeta_1, \pm \vbeta_2\}$, we may take the inner product with $\tv(\ell_i)$ on both sides of \eqref{eq:projectionKernel} to get that $s_i$ equals  either $|a c|$ or $|bc|$, and hence $s_i \in (0,1)$. Since $\tv(\ell_0) \perp \tv(\ell_1)$, this proves that $(\Pi^{\beta}_{a,b},\ell_0, s_0, \ell_1, s_1)$ is a transversal crossing. If $\vw = \vbeta_3 + a\vbeta_1 + b\vbeta_2$, the sign of this crossing is given by the sign of $\det(\tv(\ell_1), \tv(\ell_0), c\vw)$. By switching the roles of $\ell_0$ and $\ell_1$ if necessary, we may assume that $c > 0$, so that this sign equals $\det(\tv(\ell_1), \tv(\ell_0), \vw) = \det(\tv(\ell_1), \tv(\ell_0), \vbeta_3) = \tau_{a,b}(\ell_0,\ell_1)$, completing the proof. 
\end{proof}

\begin{lemma}
\label{lemma:crossings}
Let $\cR$ be a region, and let $\beta \in \bB$. Let $N$ denote the $\vbeta_3$-length of $\cR$, and suppose $0 < \epsilon < 1/N$. Given two segments $\ell_0$ and $\ell_1$,
$$\tau^{\vbeta_3}(\ell_0,\ell_1) = \frac{1}{4}\sum_{i,j \in \{-1,1\}} \tau^{\beta}_{i\epsilon, j\epsilon}(\ell_0,\ell_1) .$$
\end{lemma}  
 
\begin{proof}
We may assume that $\vbeta_3 \cdot \ell_0(0) < \vbeta_3 \cdot \ell_1(0)$ and that $\det(\tv(\ell_1), \tv(\ell_0), \vbeta_3) \neq 0$ (otherwise both sides would be zero). Since rotations in the $\vbeta_3^{\perp}$ plane leave both sides unchanged, we may assume that $\tv(\ell_1) = \pm\vbeta_1$, $\tv(\ell_0) = \pm\vbeta_2$ (see Figure \ref{fig:projectionDimerCrossings}).

\begin{figure}[ht]%
\centering
\includegraphics[width=0.9\columnwidth]{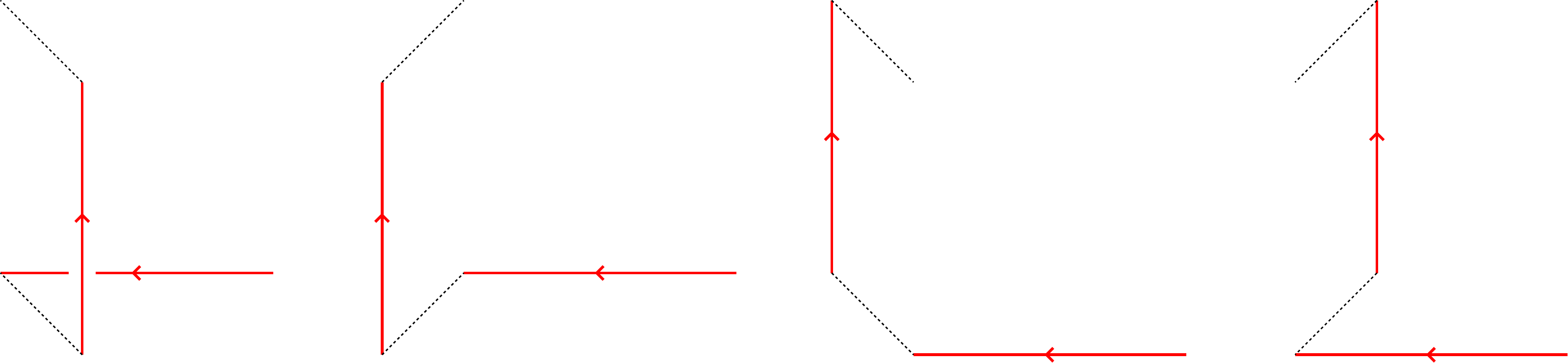}%
\caption{Illustrations of the four different projections $\Pi^{\beta}_{\pm\epsilon, \pm \epsilon}$ of two segments $\ell_0, \ell_1$ with $\tau^{\vbeta_3}(\ell_0,\ell_1) = 1/4$. The dotted lines represent the projection of lines which are parallel to $\vbeta_3$, in each of the four cases. Notice that the segments are involved in a crossing for exactly one of the projections, and this crossing is positive.}%
\label{fig:projectionDimerCrossings}%
\end{figure}

Our strategy is to show these two facts:
\begin{enumerate}[label=(\roman*)]
	\item \label{item:epsImpliesZero} If $\tau^{\beta}_{i\epsilon, j\epsilon}(\ell_0,\ell_1) \neq 0$ for some $(i,j) \in \{-1,1\}^2$, then $\tau^{\vbeta_3}(\ell_0,\ell_1) \neq 0$.
	\item \label{item:zeroImpliesEps} If $\tau^{\vbeta_3}(\ell_0,\ell_1) \neq 0$, then there exists a unique $(i,j) \in \{-1,1\}^2$ such that  $\tau^{\beta}_{i\epsilon, j\epsilon}(\ell_0,\ell_1) \neq 0$. 
\end{enumerate}
Once we prove \ref{item:epsImpliesZero} and \ref{item:zeroImpliesEps}, we get the result. 

Let $c = \vbeta_3 \cdot (\ell_1(0) - \ell_0(0))$, and consider the closed sets 
$$
A_{ij} = \left\{\delta \in [0,\epsilon] | \exists s_0,s_1 \in [0,1],  \ell_1(s_1) - \ell_0(s_0) = c(\vbeta_3 + i \delta \vbeta_1 + j \delta \vbeta_2)\right\}.
$$
Notice that $\epsilon \in A_{ij}$ if and only if $\tau^{\beta}_{i\epsilon, j\epsilon}(\ell_0,\ell_1) \neq 0$, and $0 \in A_{ij}$ if and only if $\tau^{\vbeta_3}(\ell_0,\ell_1) \neq 0$.

Suppose $\epsilon \in A_{ij}$ for some $(i,j) \in \{-1,1\}^2$, and let $\delta = \min A_{ij}$. If $\delta > 0$, $\ell_1(s_1) - \ell_0(s_0) = c(\vbeta_3 + i \delta \vbeta_1 + j \delta \vbeta_2)$ implies, by Lemma \ref{lemma:transversalCrossings},  that $s_0,s_1 \in (0,1)$. Hence, there must exist $\delta' < \delta$ such that $\delta' \in A$, a contradiction. Therefore, we must have $\delta = 0$, so that $0 \in A_{ij}$. We have proved \ref{item:epsImpliesZero}.

Now suppose $\tau^{\vbeta_3}(\ell_0,\ell_1) \neq 0$, that is, $\ell_1(k_1) - \ell_0(k_0) = c\vbeta_3$ for some $k_0,k_1 \in [0,1]$. Clearly $k_0, k_1 \in \{0,1\}$; for simplicity, assume that $k_0 = k_1 = 0$ (the other cases are analogous). Now for any $s_0, s_1 \in [0,1]$,
$\ell_1(s_1) - \ell_0(s_0) = c\vbeta_3 - s_0 \tv(\ell_0) + s_1 \tv(\ell_1)$. Thus, given $(i,j) \in \{-1,1\}^2$,
$$\epsilon \in A_{ij} \Leftrightarrow \exists s_0,s_1 \in [0,1]: \quad s_1 (\tv(\ell_1) \cdot \vbeta_1) = i \epsilon c ,\quad -s_0 (\tv(\ell_0) \cdot \vbeta_2) = j \epsilon c,$$
which occurs if and only if $i \epsilon c (\tv(\ell_1) \cdot \vbeta_1) > 0$ and $j \epsilon c (\tv(\ell_0) \cdot \vbeta_2) < 0$: this determines a unique $(i,j) \in \{-1,1\}^2$, so we have proved \ref{item:zeroImpliesEps}.
\end{proof}

If $A_0$ and $A_1$ are two sets of segments (curves are also seen as sets of segments), $\vu \in \Phi, \beta \in \bB$, define
$$T^{\vu}(A_0,A_1) = \sum_{\substack{\ell_0 \in A_0 \\ \ell_1 \in A_1}} \tau^{\vu}(\ell_0,\ell_1), \quad T^{\beta}_{a,b}(A_0,A_1) = \sum_{\substack{\ell_0 \in A_0 \\ \ell_1 \in A_1}} \tau^{\beta}_{a,b}(\ell_0,\ell_1).$$
For shortness, $T^{\vu}(A) = T^{\vu}(A,A)$ and $T^{\beta}_{a,b}(A) = T^{\beta}_{a,b}(A,A)$.

Consider two disjoint simple closed curves $\gamma_0, \gamma_1$ and a projection $\Pi$ from $\RR^3$ to some plane. Assume there exists finitely many crossings $(\Pi,\gamma_0, s_0, \gamma_1, s_1)$, all transversal.
Recall from knot theory (see, e.g., \cite[pp. 18--19]{knotbook}) that the \emph{linking number} $\Link(\gamma_0,\gamma_1)$ equals half the sum of the signs of all these crossings.

\begin{lemma}
\label{lemma:linkingNumber}
Let $\gamma_0$ and $\gamma_1$ be two disjoint simple closed curves of a region $\cR$. Fix $\beta \in \bB$, and let $N$ denote the $\vbeta_3$-length of $\cR$. Then
\begin{enumerate}[label=\upshape(\roman*),topsep=0.1em]
\item \label{item:linking_ab} If $0 < |a|,|b| < 1/N$, $T^{\vbeta}_{a,b}(\gamma_0, \gamma_1) + T^{\vbeta}_{a,b}(\gamma_1, \gamma_0) = 2\Link(\gamma_0,\gamma_1)$. 
\item \label{item:linking_00} $T^{\vbeta_3}(\gamma_0,\gamma_1) + T^{\vbeta_3}(\gamma_1,\gamma_0) = 2\Link(\gamma_0,\gamma_1).$
\end{enumerate}
\end{lemma}

\begin{proof}
By Lemma \ref{lemma:transversalCrossings}, the sum of signs of the crossings is given by $T^{\beta}_{a,b}(\gamma_0,\gamma_1) + T^{\beta}_{a,b}(\gamma_1,\gamma_0)$, which establishes \ref{item:linking_ab}. Also, \ref{item:linking_00} follows from \ref{item:linking_ab} and Lemma \ref{lemma:crossings}. 
\end{proof}


\begin{lemma}
\label{lemma:intersectionOfSegments}
Let $\ell_0$ and $\ell_1$ be two segments of $\cR$, and let $\vu \in \RR^3$ be a vector such that $\|\vu\| < 1$. Then these two statements are equivalent:
\begin{enumerate}[label=\upshape(\roman*)]
	\item \label{item:intersection} There exist $s_0,s_1 \in [0,1]$ such that $\ell_0(s_0) - \ell_1(s_1) = \vu$.
	\item \label{item:conditions} There exist $(i,j) \in \{0,1\}^2$ and $a_0,a_1 \in (-1,1)$ such that
	$\ell_0(i) = \ell_1(j)$ and $\vu = a_0 \tv(\ell_0) + a_1 \tv(\ell_1)$ with $(-1)^i a_0 \geq 0$ and $(-1)^j a_1 \leq 0$.
\end{enumerate}
\end{lemma}
\begin{proof}
First, suppose \ref{item:intersection} holds. If $\ell_0$ and $\ell_1$ are not adjacent, then $\dist(\ell_0,\ell_1) \geq 1 > \|\vu\|$, which is a contradiction. Thus, $\ell_0$ and $\ell_1$ are adjacent, and thus $\ell_0(i) = \ell_1(j)$ for some $(i,j) \in \{0,1\}^2$: then 
\begin{align*}
\vu = \ell_0(s_0) - \ell_1(s_1) &= [\ell_0(i) + (s_0 - i) \tv(\ell_0)] - [\ell_1(j) + (s_1 - j)\tv(\ell_1)]\\
 &= (s_0 - i) \tv(\ell_0) + (j-s_1)\tv(\ell_1),
\end{align*}
that is, $\vu = a_0 \tv(\ell_0) + a_1 \tv(\ell_1)$ with $(i + a_0), (j - a_1) \in [0,1]$, which implies that $(-1)^i a_0 \geq 0$ and $(-1)^j a_1 \leq 0$. Also, since $\|\vu\| < 1$, we can take $a_0,a_1 \in (-1,1)$.

For the other direction, suppose \ref{item:conditions} holds, so that $\ell_0(i) = \ell_1(j)$ for some $(i,j) \in \{0,1\}^2$. Then setting  $s_0 = (i + a_0)$ and $s_1 = (j - a_1)$, we have $s_0,s_1 \in [0,1]$ and 
$\ell_0(i + a_0) - \ell_1(j - a_1) = [\ell_0(i) + a_0\tv(\ell_0)] - [\ell_1(j) - a_1\tv(\ell_1)] = \vu. $   
\end{proof}

For a map $\gamma:[0,n] \to \RR^3$ and a vector $\vu \in \RR^3$, let $(\gamma + \vu):[0,1] \to \RR^3: s \mapsto \gamma(s) + \vu$ denote the translation of $\gamma$ by $\vu$. 

\begin{lemma}
\label{lemma:translationIntersection}
Let $\gamma$ be a curve of $\cR$, let $\beta \in \bB$, and let $\vu = a \vbeta_1 + b \vbeta_2 + c \vbeta_3 \in \RR^3$. If $\|\vu\| < 1$ and $abc \neq 0$, then the curves $\gamma$ and $\gamma + \vu$ are disjoint.  
\end{lemma}
Notice that $\gamma + \vu$ is not a curve of $\cR$.
\begin{proof}
Suppose, by contradiction, that there exist $s_0, s_1 \in [0,n]$ (the domain of $\gamma$) such that $\gamma(s_0) = \gamma(s_1) + \vu.$ Let $k_0, k_1 \in \ZZ$ be such that $k_i \leq s_i \leq k_i + 1 \leq n$, and set $\tilde{s}_i = s_i - k_i$. Since $\gamma$ is a curve of $\cR$, $\ell_i = \gamma|_{[k_i,k_i+1]}$ are segments of $\cR$ such that $\ell_0(\tilde{s}_0) - \ell_1(\tilde{s}_1) = \gamma(s_0) - \gamma(s_1) = \vu$. By Lemma \ref{lemma:intersectionOfSegments}, $\vu = a_0 \tv(\ell_0) + a_1 \tv(\ell_1)$, which means that at least one of the three coordinates of $\vu$ is zero: this contradicts the fact that $abc \neq 0$. 
\end{proof}

Consider a simple closed curve $\gamma: I \to \RR^3$ and a vector $\vu \in \RR^3$, $\vu \neq 0$. Assume that there exists $\delta > 0$ such that for each $s \in (0,\delta]$, the curves $\gamma$ and $\gamma + s\vu$ are disjoint. Then define the \emph{directional writhing number} in the direction $\vu$ by $\Wr(\gamma, \vu) = \Link(\gamma, \gamma + \delta\vu)$ (see \cite[\S 3]{writhingNumber}). Since $\Link$ is symmetric and invariant by translations, $\Wr(\gamma,\vu) = \Wr(\gamma,-\vu)$.

\begin{lemma}
\label{lemma:directionalWrithingNumber}
Fix $\beta \in \bB$, and let $\gamma$ be a simple closed curve of $\cR$. If $0 < |a|,|b| < 1/N$, where $N$ is the $\vbeta_3$-length of $\cR$, then 
$\Wr(\gamma, \vbeta_3 + a\vbeta_1 + b\vbeta_2) = T^{\beta}_{a,b}(\gamma).$
\end{lemma}
\begin{proof}
We would like to use the fact that the sums of the signs of the crossings of the orthogonal projection of a smooth curve in the direction of a vector $\vu$ equals its directional writhing number (in the direction of $\vu$): this is essentially what we're trying to prove for our curve, except that $\Pi^{\beta}_{a,b}$ is not the orthogonal projection and that $\gamma$ is not a smooth curve. However, these difficulties can be avoided, as the following paragraphs show.

The orthogonality of the projection makes no real difference, because the orthogonal projection in the direction of $(a,b,1)$ has the same kernel as $\Pi^{\beta}_{a,b}$, so the crossings occur in the same positions (and clearly have the same signs). Therefore, by Lemma \ref{lemma:transversalCrossings}, $T^{\beta}_{a,b}(\gamma)$ equals the sums of the signs of the crossings of the aforementioned orthogonal projection.

For the smoothness of the curve, there is a finite number of points where $\gamma$ is not smooth: precisely, the set of $k \in \ZZ$ such that the two segments of $\gamma$ that intersect at $\gamma(k)$ are not parallel. To simplify notation, let $[0,n]$ be the domain of $\gamma$, and for $k=0,1,\ldots,n-1$ let $\ell_k$ be the segment of $\gamma$ such that $\ell_k(0) = \gamma(k)$ (notice that $\ell_k(1) = \gamma(k+1)$). It is also convenient to set $\ell_{-1} := \ell_{n-1}$, so that $\ell_{-1}(1) = \ell_{n-1}(1) = \gamma(n) = \gamma(0)$.

Recall from Lemma \ref{lemma:transversalCrossings} that every crossing in the projections occur in the interiors of the segments: since the number of segments is finite, we can pick $0 < \epsilon < 1/2$ sufficiently small so that 
$\Pi^{\beta}_{a,b}(\gamma(U_{\epsilon}))$ 
contains no crossings, where 
$U_\epsilon = [0,n] \cap \left(\bigcup_{k \in \ZZ} [k - \epsilon, k + \epsilon]\right).$

Let $\phi_1: \RR \to \RR$ be a nondecreasing $C^{\infty}$ function such that $\phi_1(t) = 0$ whenever $t \leq -\epsilon$ and $\phi_1(t) = t$ whenever $t \geq \epsilon$.
Let $\phi_0(t) = t + \epsilon - \phi_1(t)$.
Consider the smooth simple closed curve of $\RR^3$, $\tilde{\gamma}:[0,n] \to \RR^3$, given by
$$ \tilde{\gamma}(s) = 
\begin{cases}
\gamma(k - \epsilon) + \phi_0(s-k)\tv(\ell_{k-1}) + \phi_1(s-k)\tv(\ell_k), &s \in (k-\epsilon, k+ \epsilon);\\
\gamma(s), &s \notin U_{\epsilon}.
\end{cases}$$
To simplify notation, write $\vw = \vbeta_3 + a\vbeta_1 + b\vbeta_2$ and fix $\delta < 1/\sqrt{1 + a^2 + b^2}$, so that $\|\delta\vw\| < 1$. By Lemma \ref{lemma:translationIntersection}, $\gamma$ and $\gamma + s\vu$ are disjoint whenever $s \in (0,\delta]$.
  
Clearly, the sums of the signs of the crossings in the orthogonal projection of $\tilde{\gamma}$ equals that of $\gamma$; moreover, $\Link(\tilde{\gamma},\tilde{\gamma} +  s\vw) = \Link(\gamma,\gamma +  s\vw)$ for sufficiently small $s > 0$. Since $\tilde{\gamma}$ is smooth, $T^{\beta}_{a,b}(\gamma) = \Wr(\tilde{\gamma},\vw) = \Link(\gamma,\gamma +  s\vw) = \Wr(\gamma,\vw)$. 
%
%
\end{proof}

The following rather technical Lemma will be used in the proof of Lemma \ref{lemma:differenceOfLinkingNumbers}:

\begin{lemma}
\label{lemma:projectionIntersections}
Let $\beta \in \bB$, and let $\ell_0$ and $\ell_1$ be two segments of a region $\cR$ whose $\vbeta_3$-length is $N$. Let $\vu = b \vbeta_2 + c \vbeta_3$ with
$bc \neq 0$ and $b^2 + c^2 < 1$.
Let $0 < \epsilon < \min\left(\frac{|b|}{N + |c|}, \frac{1 - |b|}{N + |c|}\right)$. 

If, for some $s_0, s_1 \in [0,1]$,
$\Pi^{\beta}_{\epsilon,\epsilon}(\ell_0(s_0) - \ell_1(s_1) - \vu) = 0$,
then $\ell_0$ and $\ell_1$ are not parallel, and $s_0,s_1 \in (0,1)$. 
\end{lemma}
\begin{proof}
Suppose $\Pi^{\beta}_{\epsilon,\epsilon}(\ell_0(s_0) - \ell_1(s_1) - \vu) = 0$. Let $\alpha_i = \vbeta_i \cdot (\ell_0(s_0) - \ell_1(s_1)), i=1,2,3$, so that $\alpha_1 = \epsilon (\alpha_3 - c), \alpha_2 - b = \epsilon (\alpha_3 - c)$.
 
Suppose, by contradiction, that at least one of these things occurs: 
\begin{enumerate}[label=(\roman*), topsep = 0.1px, itemsep = 0.1px]
	\item \label{item:parallel} $\ell_0$ and $\ell_1$ are parallel;
	\item \label{item:borderline} $s_0 \in \{0,1\}$ or $s_1 \in \{0,1\}$.
\end{enumerate}
We claim that at least two of the three $\alpha_i$'s are integers. To see this, suppose first \ref{item:parallel}, so that $\tv(\ell_0), \tv(\ell_1) \parallel \vbeta_i$, so that for $j \neq i$, $\alpha_j = \vbeta_j \cdot (\ell_0(s_0) - \ell_1(s_1)) \in \ZZ$. On the other hand, if \ref{item:borderline} holds, say $s_1 \in \{0,1\}$, then $\ell_1(s_1) \in \plshalf{\ZZ}$ and $\tv(\ell_0) \parallel \vbeta_i$, so that, again, for $j \neq i$, $\alpha_j \in \ZZ$.

We claim that $\alpha_2 \notin \ZZ$. In fact, if $\alpha_2 \in \ZZ$ then we would have
$|\alpha_2| = |b + \epsilon(\alpha_3 - c)| < |b| + \frac{1 - |b|}{N + |c|}(N + |c|) = 1$, so that $\alpha_2 = 0$ and $|b| = \epsilon|\alpha_3 - c| < \frac{|b|}{N + |c|}(N + |c|)$, which is a contradiction.

Therefore, we must have $\alpha_1,\alpha_3 \in \ZZ.$ Then $|\alpha_1| = |\epsilon(\alpha_3 - c)| < |b| < 1$, so $\alpha_1 = 0 = \alpha_3 - c$. Thus $c = \alpha_3 \in \ZZ$ but $|c| \in (0,1)$, which is a contradiction.
%
%
%
\end{proof}

The following definition is specific for Lemmas \ref{lemma:differenceOfLinkingNumbers} and \ref{lemma:writheDifference}. Let $\gamma:[0,n] \to \RR^3$ be a simple closed curve of a region $\cR$ and $\beta \in \bB$. For $k=0,1,\ldots, n-1$, set $\ell_k = \gamma|_{[k,k+1]}$, and set also $\ell_n = \ell_0$. Finally, we define 
$$\eta^{\beta}_{\gamma}(k) = 
\begin{cases}
1, &(\tv(\ell_k), \tv(\ell_{k+1})) = (\vbeta_2,\vbeta_3) \text{ or } (-\vbeta_3,-\vbeta_2);\\
-1, & (\tv(\ell_k), \tv(\ell_{k+1})) = (-\vbeta_2,-\vbeta_3) \text{ or } (\vbeta_3,\vbeta_2);\\
0, &\text{otherwise.}
\end{cases}
$$

\begin{lemma}
\label{lemma:differenceOfLinkingNumbers}
Let $\gamma:[0,n] \to \RR^3$ be a simple closed curve of a region $\cR$. For $k=0,1,\ldots, n-1$, set $\ell_k = \gamma|_{[k,k+1]}$, and set also $\ell_n = \ell_0$; for shortness, write $\tv_k = \tv(\ell_k)$. Then if $\beta \in \bB$ and $a,b,c > 0$, then
$$
\Wr(\gamma, a \vbeta_1 + b \vbeta_2 + c \vbeta_3) - \Wr(\gamma, - a \vbeta_1 + b \vbeta_2 + c \vbeta_3) = \sum_{0 \leq k < n} \eta^{\beta}_{\gamma}(k).
$$
\end{lemma}
\begin{proof}
We may assume that $a^2 + b^2 + c^2 < 1$.
Let $0 < \epsilon < \min\left(\frac{|b|}{N + |c|}, \frac{1 - |b|}{N + |c|}\right),$ and set $\vu(s) = s \vbeta_1 + b \vbeta_2 + c \vbeta_3$.
By Lemma \ref{lemma:translationIntersection}, $\Link(\gamma, \gamma + \vu(a))$ depends only on the signs of $a$, $b$ and $c$. Therefore, we may, without loss of generality, assume that $a > 0$ is sufficiently small such that for every $s \in [-a,a]$ and every $i,j \in \{0,1, \ldots, n\}$,
$$\Pi^{\beta}_{\epsilon,\epsilon}(\ell_i) \cap \Pi^{\beta}_{\epsilon,\epsilon}(\ell_j + \vu(s)) \neq \emptyset \Leftrightarrow \Pi^{\beta}_{\epsilon,\epsilon}(\ell_i) \cap \Pi^{\beta}_{\epsilon,\epsilon}(\ell_j + \vu(0)) \neq \emptyset$$
(this is possible by Lemma \ref{lemma:projectionIntersections}). Therefore, clearly $\Link(\gamma, \gamma + \vu(a)) - \Link(\gamma, \gamma + \vu(-a))$ equals the number of pairs of segments $\ell_i, \ell_j$ such that $\Pi^{\beta}_{\epsilon,\epsilon}(\ell_i) \cap \Pi^{\beta}_{\epsilon,\epsilon}(\ell_j + \vu(s)) \neq \emptyset$ for every $s \in [-a,a]$ and such that the crossing changes its sign as $s$ goes from $-a$ to $a$. Now a crossing may only change its sign if $\ell_i \cap (\ell_j + \vu(s)) \neq \emptyset$ for some $s$: by Lemma \ref{lemma:translationIntersection}, this can only happen if $s = 0$. 

By Lemma \ref{lemma:intersectionOfSegments}, $\ell_i \cap (\ell_j + \vu(0)) \neq \emptyset$ if and only if for some $m_i,m_j \in \{0,1\}$, $\ell_i(m_i) = \ell_j(m_j)$ and
$\vu(0) = b \vbeta_2 + c \vbeta_3 = a_i \tv(\ell_i) + a_j \tv(\ell_j)$, with $(-1)^{m_i} a_i \geq 0, (-1)^{m_j} a_j \leq 0$.
Since $\ell_i$ and $\ell_j$ are segments of the simple curve $\gamma$, they can only be adjacent if, for some $k$, $\{\ell_i, \ell_j\} = \{\ell_k, \ell_{k+1}\}$. Now, $\ell_{k+1}(0) = \ell_k(1)$, so that 
$(0,b,c) = a_0\tv(\ell_{k+1}) - a_1\tv(\ell_k)$ with either $a_0,a_1 \geq 0$ or $a_0,a_1 \leq 0$ (depending on which is $\ell_i$ and which is $\ell_j$). Since $b,c > 0$, this implies that $\{\tv(\ell_k), \tv(\ell_{k+1})\} = \{\vbeta_2,\vbeta_3\} \text{ or } \{-\vbeta_2, -\vbeta_3\}$ and, therefore, $\eta^{\beta}_{\gamma}(k) = \pm 1$.
\begin{figure}[ht]
\centering
\def\svgwidth{0.35\columnwidth}
\subfloat[$\eta^{\beta}_{\gamma}(k) = 1$.]{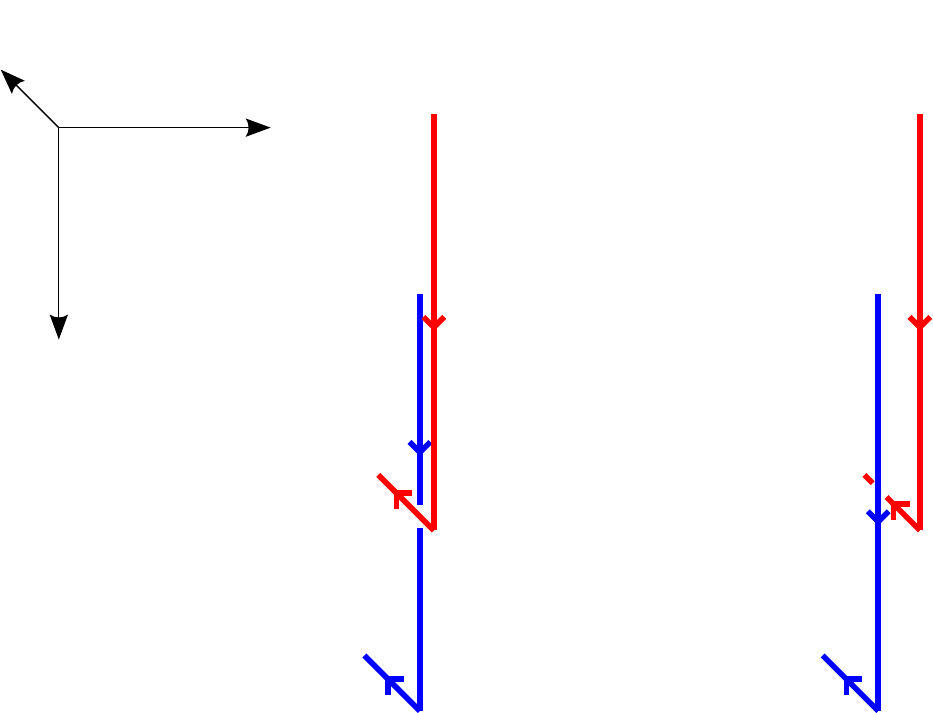\label{subfig:crossings_jk}} \qquad \qquad \qquad \qquad 
\def\svgwidth{0.35\columnwidth}
\subfloat[$\eta^{\beta}_{\gamma}(k) = -1$.]{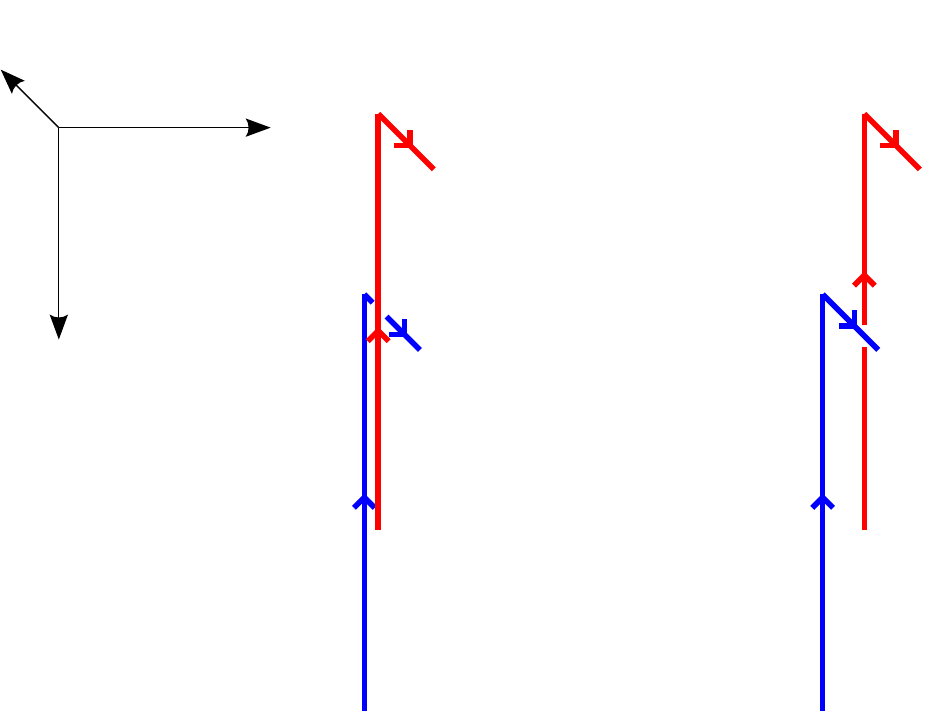\label{subfig:crossings_kj}}
\caption{Illustration of the crossings $\Pi^{\beta}_{\epsilon,\epsilon}(\gamma) \cap \Pi^{\beta}_{\epsilon,\epsilon}(\gamma+s\vbeta_1 + b \vbeta_2 + c \vbeta_3)$ for $s \in [-a,a]$. Notice that simultaneously switching the orientations of both segments does not change the signs of the crossings.}%
\label{fig:cases_YDimers}%
\end{figure}

We now analyze each of the four possible cases for $(\tv(\ell_k), \tv(\ell_{k+1}))$ (as an ordered pair). When $(\tv(\ell_k), \tv(\ell_{k+1})) = (\vbeta_2,\vbeta_3) \text{ or } (-\vbeta_3,-\vbeta_2)$, so that $\eta^{\beta}_{\gamma}(k) = 1,$ we see a situation as illustrated in Figure \ref{subfig:crossings_jk} (perhaps with both orientations reversed): when $s > 0$, we have a positive crossing; when $s < 0$, we have a negative crossing.
Figure \ref{subfig:crossings_kj} illustrates (up to orientation) the case $(\tv(\ell_k), \tv(\ell_{k+1})) = (-\vbeta_2,-\vbeta_3) \text{ or } (\vbeta_3,\vbeta_2)$ ($\eta^{\beta}_{\gamma}(k) = -1)$: negative crossing for $s > 0$, and positive crossing for $s < 0$. These observations yield the result.
\end{proof}

\section{Writhe formula for the twist}
\label{sec:writheInterpretation}
 
Now that the groundwork is done, we set out to obtain a new formula for the twist of pseudocylinders of even depth (we work with pseudocylinders because the hypothesis of simple connectivity will not play any role). Pseudocylinders of even depth have the advantage of always admitting a tiling such that all dimers are parallel to its axis: for a $\vw$-pseudocylinder $\cR$ ($\vw \in \Delta$) with even depth, let $\tbase = \tbase(\cR)$ denote the tiling such that every dimer is parallel to $\vw$ (see Figure \ref{fig:auxiliaryLinesExample}). Not only does this tiling trivially satisfy $T^{\vw}(\tbase) = 0$, but also for any segment $\ell$ of $\cR$ and any dimer $\ell_0 \in \tbase$ we have $\tau^{\vw}(\ell_0,\ell) = \tau^{\vw}(\ell, \ell_0) = 0$. This allows for a direct interpretation of the twist via a set of curves, which, in particular, allows us to show that it is an integer.

\begin{figure}[ht]%
\centering
\subfloat[The tiling $t$.]{\includegraphics[width=0.45\columnwidth]{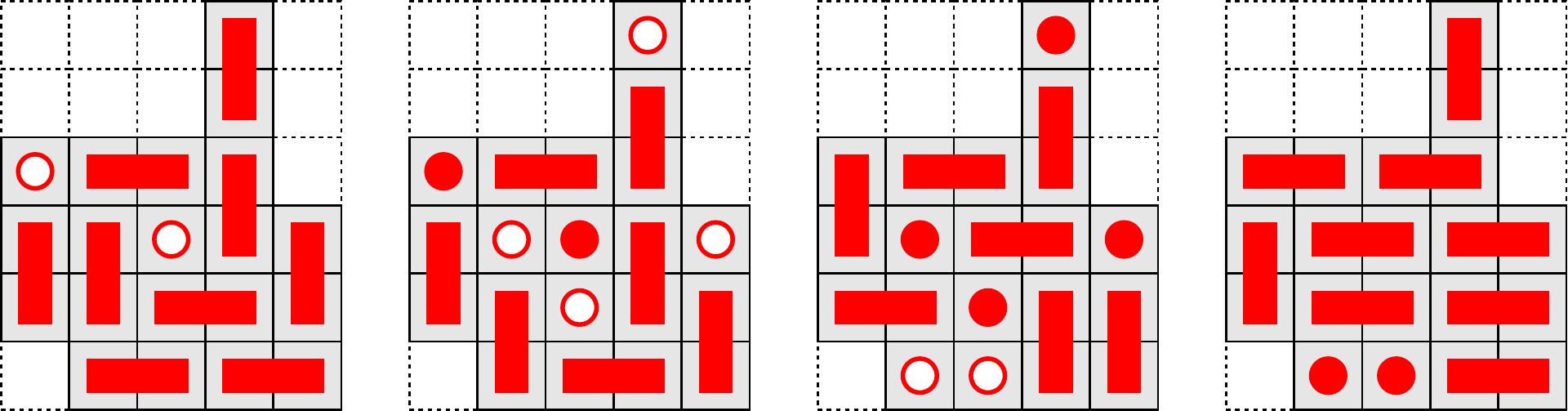}} \qquad
\def\svgwidth{0.47\columnwidth}
\subfloat[The curves in $\Gamma(t,\tbase)$.]{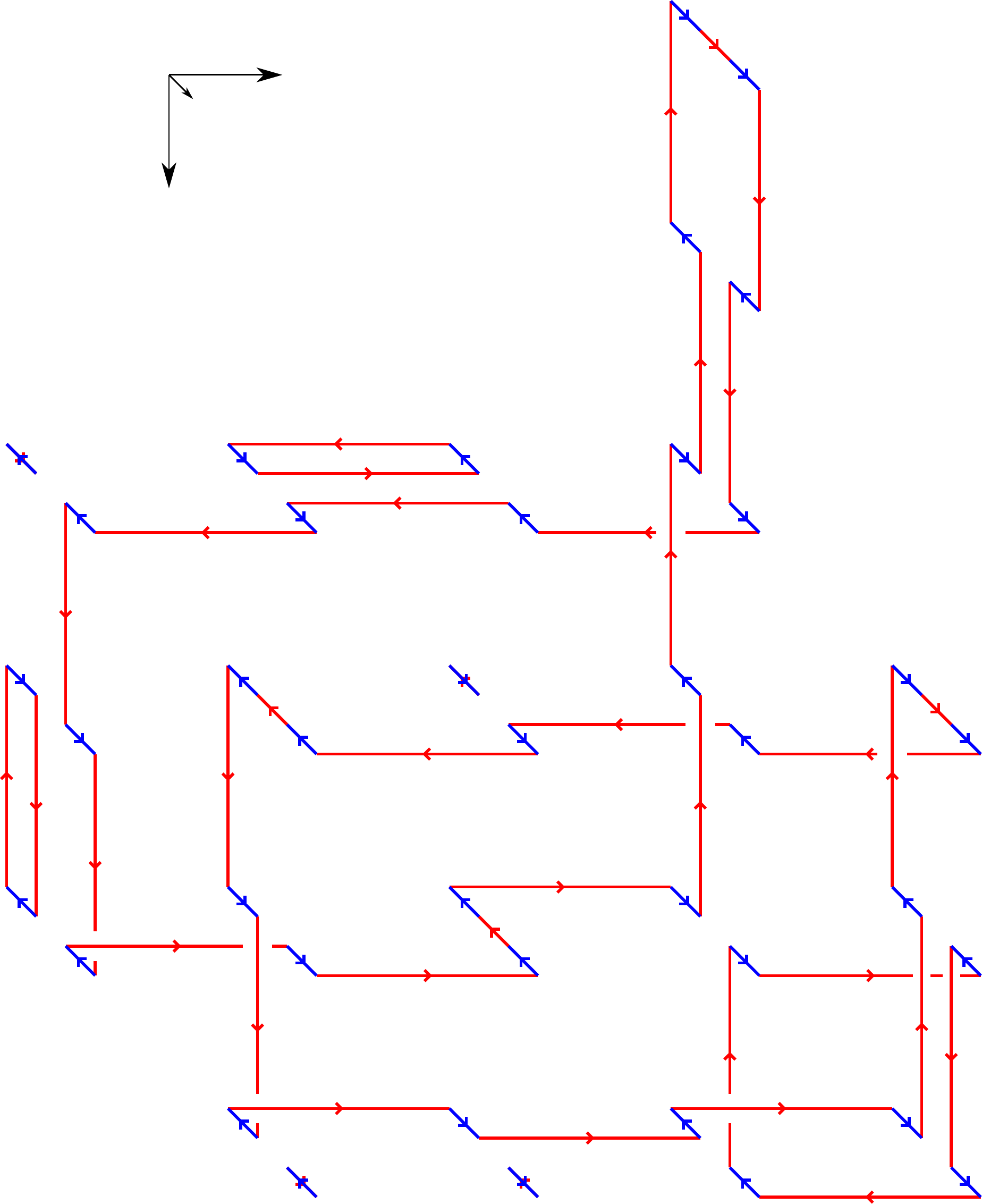\label{subfig:auxiliaryLinesDrawing}}%
\caption{A tiling $t$ of a $\vw$-cylinder with depth $4$, and $\Gamma(t,\tbase)$, where $\tbase$ is the tiling such that every dimer is parallel to $\vw$. The dimers of $t$ are the red segments, and the blue segments are the ones in $(-\tbase)$. We chose a basis $\beta \in \bB$ with $\vw = \vbeta_3$; $\vw$ points ``towards the paper''. $\Gamma(t,\tbase)$ consists of nine curves, four of which are trivial; the five nontrivial curves form $\Gamma^*(t,\tbase)$.}
\label{fig:auxiliaryLinesExample}%
\end{figure}

\begin{lemma}
\label{lemma:partialWritheFormula}
Given $\vw \in \Delta$, let $t$ be a tiling of a $\vw$-pseudocylinder of even depth $\cR$, and let $\tbase = \tbase(\cR)$. If $\Gamma^*(t, \tbase) = \{\gamma_i \,|\, 1 \leq i \leq m\}$, then 
$$T^{\vw}(t) = \sum_{1 \leq i \leq m} T^{\vw}(\gamma_i) + 2\sum_{1 \leq i < j \leq m} \Link(\gamma_i, \gamma_j).$$
\end{lemma}
\begin{proof}
Clearly,  
$$
T^{\vw}(t) = T^{\vw}(t \sqcup (-\tbase))  
= \sum_{i,j} T^{\vw}(\gamma_i, \gamma_j)  
= \sum_{i} T^{\vw}(\gamma_i) + 2\sum_{i < j} \Link(\gamma_i, \gamma_j),
$$
the last equality holding by Lemma \ref{lemma:linkingNumber}.
\end{proof}



For Lemmas \ref{lemma:gammaNoIntersection} and \ref{lemma:writhes}, assume $\vw \in \Delta$, $t$ is a tiling of a $\vw$-pseudocylinder with even depth $\cR$, and $\tbase = \tbase(\cR)$.

\begin{lemma}
\label{lemma:gammaNoIntersection}
Fix $\beta \in \bB$ such that $\vbeta_3 = \vw$.
If $\gamma$ is a curve of $\Gamma^*(t,\tbase)$ and $a^2 + b^2 + c^2 < 1$ and $ab \neq 0$, then $(\gamma+a \vbeta_1 + b \vbeta_2 + c \vbeta_3) \cap \gamma = \emptyset$.
\end{lemma}
Notice that the case $c \neq 0$ follows from Lemma \ref{lemma:translationIntersection}.
\begin{proof} 
Let $\vu = a \vbeta_1 + b \vbeta_2 + c \vbeta_3.$
Suppose, by contradiction, that $\gamma$ and $\gamma + \vu$ are not disjoint, and let 
$\ell_0$ and $\ell_1$ be two segments of $\gamma$ such that $\ell_0(s_0) = \ell_1(s_1) + \vu$ for some $s_0,s_1 \in [0,1]$. By Lemma \ref{lemma:intersectionOfSegments}, $\ell_0$ and $\ell_1$ must be adjacent, so that at least one of these two segments is in $(-\tbase)$, hence parallel to $\vu$. Lemma \ref{lemma:intersectionOfSegments} also implies that $\vu = a \vbeta_1 + b \vbeta_2 + c \vbeta_3 = a_0 \tv(\ell_0) + a_1 \tv(\ell_1)$. Since at least one of $\tv(\ell_0), \tv(\ell_1)$ is parallel to $\vw = \vbeta_3$, it follows that $a = 0$ or $b = 0$, which contradicts the hypothesis.  
\end{proof}

By Lemma \ref{lemma:gammaNoIntersection}, if $\gamma \in \Gamma^*(t,\tbase)$, $\Wr(\gamma, a\vbeta_1 + b\vbeta_2 + c\vbeta_3)$ is defined whenever $ab \neq 0$. Set
$$\Wr^{+}(\gamma) = \Wr(\gamma, \vbeta_1 + \vbeta_2), \quad \Wr^{-}(\gamma) = \Wr(\gamma, \vbeta_1 - \vbeta_2).$$
Clearly, 
\begin{equation}
\Wr(\gamma, a \vbeta_1 + b \vbeta_2 + c \vbeta_3) = \begin{cases} \Wr^+(\gamma), & ab > 0; \\
\Wr^-(\gamma), & ab < 0.
\end{cases}
\label{eq:writhes}
\end{equation}

\begin{lemma}
\label{lemma:writhes}
If $\gamma$ is a curve of $\Gamma^*(t,\tbase)$, then
$$T^{\vw}(\gamma) = \frac{\Wr^+(\gamma) + \Wr^-(\gamma)}{2}.$$
\end{lemma}
\begin{proof}
Fix $\beta \in \bB$ with $\vbeta_3 = \vw$, and let $N$ denote the $\vw$-length of the pseudocylinder (which is equal to its depth).
By Lemmas \ref{lemma:crossings} and \ref{lemma:directionalWrithingNumber}, given $0 < \epsilon < 1/N,$
$$
T^{\vw}(\gamma) = \frac{1}{4}\sum_{i,j \in \{-1,1\}} T^{\beta}_{i\epsilon,j\epsilon}(\gamma) = \frac{1}{4}\sum_{i,j \in \{-1,1\}} \Wr(\gamma, i\epsilon\vbeta_1 + j \epsilon \vbeta_2 + \vbeta_3);  
$$
Equation \eqref{eq:writhes} completes the proof.
\end{proof}


\begin{lemma}
\label{lemma:writheDifference}
Let $\vw \in \Delta$, and let $t$ be a tiling of a $\vw$-pseudocylinder with even depth.
If $\gamma$ is a curve of $\Gamma^*(t,\tbase)$, then $(\Wr^+(\gamma) + \Wr^-(\gamma))/2 \in \ZZ$.
\end{lemma}
\begin{proof}
Pick $\beta \in \bB$ with $\vbeta_3 = \vw$. Assume without loss of generality that $\cR = \cD + [0,2N]\vbeta_3$, $\cD \subset \vbeta_3^{\perp}$. 
If $\gamma:[0,n] \to \RR^3$, set $\ell_k = \gamma|_{[k,k+1]}$ for $k=0,1,\ldots, n-1$, and set $\ell_n = \ell_0$. 

By definition and using Lemma \ref{lemma:differenceOfLinkingNumbers},
$
\Wr^+(\gamma) - \Wr^-(\gamma) = \sum_k \eta^{\beta}_{\gamma}(k).
$
We need to look at $k$ such that $\eta^{\beta}_{\gamma}(k) \neq 0$, i.e., $\{\tv(\ell_k),\tv(\ell_{k+1})\} = \{\vbeta_2,\vbeta_3\} \text{ or } \{-\vbeta_2,-\vbeta_3\}$. Since every segment of $-\tbase$ is parallel to $\vbeta_3$, we need to look at every segment of $t$ that is parallel to $\vbeta_2$.

For each segment $\ell_k$ of $t$ with $\tv(\ell_k) = \pm \vbeta_2$, let $\plshalf{z_k} = \vbeta_3 \cdot \ell_k(0)$, so that $z_k \in \ZZ$. If $z_k$ is odd, then, by definition of $\tbase$, $\tv(\ell_{k-1}) = \vbeta_3 = - \tv(\ell_{k+1})$, so that either $(\tv(\ell_{k-1}),\tv(\ell_k)) = (\vbeta_3,\vbeta_2)$ or $(\tv(\ell_k),\tv(\ell_{k+1})) = (-\vbeta_2,-\vbeta_3)$. Making a similar analysis for $z_k$ even, we see that $\eta^{\beta}_{\gamma}(k-1) + \eta^{\beta}_{\gamma}(k) = (-1)^{z_k}$. 
Working with congruence modulo $2$,
\begin{equation*}
\Wr^+(\gamma) + \Wr^-(\gamma) \equiv \Wr^+(\gamma) - \Wr^-(\gamma) = \sum_{\tv(\ell_k) = \pm \vbeta_2}(-1)^{z_k} \equiv \sum_{k} (\tv(\ell_k) \cdot \vbeta_2) = 0, 
\end{equation*}
which completes the proof.
\end{proof}

\begin{prop}
\label{prop:writheFormula}
If $\vw \in \Delta$, $\cR$ is a $\vw$-pseudocylinder with even depth, $t$ is a tiling of $\cR$, $\tbase = \tbase(\cR)$ and $\Gamma^*(t,\tbase) = \{\gamma_i \,|\, 1 \leq i \leq m\}$, then
$$
T^{\vw}(t) = \sum_{1 \leq i \leq m} \frac{\Wr^+(\gamma_i) + \Wr^-(\gamma_i)}{2} + 2\sum_{1 \leq i < j \leq m} \Link(\gamma_i, \gamma_j) \in \ZZ.
$$
\end{prop}
\begin{proof}
Follows directly from Lemmas \ref{lemma:partialWritheFormula}, \ref{lemma:writhes} and \ref{lemma:writheDifference}.
\end{proof}

\section{Different directions of projection}
\label{sec:differentDirections}

Our goal for this Section is to prove Proposition \ref{prop:equalTwistsMultiplex}, that is, that all pretwists coincide for a cylinder. 

\begin{lemma}
\label{lemma:xEffectsEqualsTwist}
Let $\vw \in \Delta$, and let $\cR$ be a $\vw$-pseudocylinder with even depth.
Let $t$ be a tiling of $\cR$, and let $\tbase$ be the tiling such that every dimer is parallel to $\vw$. If $\vu \in \Phi$, then 
$T^{\vu}(t \sqcup (-\tbase)) = T^{\vw}(t).$ 
\end{lemma}
\begin{proof}
Suppose $\Gamma^*(t,\tbase) = \{\gamma_i \,|\, 1 \leq i \leq m\}$.  
Clearly, 
$$
T^{\vu}(t \sqcup (-\tbase)) = 
\sum_{i,j} T^{\vu}(\gamma_i,\gamma_j)
= \sum_i T^{\vu}(\gamma_i) + 2\sum_{i < j} \Link(\gamma_i, \gamma_j).
$$
Let $L$ be the $\vu$-length of $\cR$ and $0 < \epsilon < 1/L$. Let $\beta \in \bB$ such that $\vbeta_3 = \vu$. Then, by Lemmas \ref{lemma:crossings} and \ref{lemma:directionalWrithingNumber}, 
$$
T^{\vu}(\gamma_i) = \frac{1}{4}\sum_{k,l \in \{-1,1\}}T^{\beta}_{k \epsilon, l \epsilon}(\gamma_i) = \frac{1}{4}\sum_{k,l \in \{-1,1\}}\Wr(\gamma_i, k\epsilon \vbeta_1 + l\epsilon \vbeta_2 + \vbeta_3).
$$
By Equation \eqref{eq:writhes} and Proposition \ref{prop:writheFormula}, 
$$
T^{\vu}(t \sqcup (-\tbase)) = \sum_i \frac{\Wr^+(\gamma_i) + \Wr^-(\gamma_i)}{2} + 2\sum_{i < j} \Link(\gamma_i, \gamma_j) = T^{\vw}(t).
$$
\end{proof}

\begin{lemma}
\label{lemma:allPretwistsBoxes}
Let $\cB = [0,L] \times [0,M] \times [0,N]$ be a box that has at least one even dimension, and let $t$ be a tiling of $\cB$. Then $T^{\ex}(t) = T^{\ey}(t) = T^{\ez}(t)$.
\end{lemma}

\begin{proof}
By rotating, we may assume that $N$ is even, so that $\cB$ is a $\ez$-cylinder with even depth; let $\vu \in \Phi, \vu \perp \ez$. We want to show that $T^{\vu}(t) = T^{\ez}(t)$.

By Lemma \ref{lemma:xEffectsEqualsTwist}, 
$T^{\ez}(t) = T^{\vu}(t \sqcup (-\tbasez))$.
%
%
Now, 
$$T^{\vu}(t \sqcup (-\tbasez)) = T^{\vu}(t) + T^{\vu}(-\tbasez) + T^{\vu}(t, -\tbasez) + T^{\vu}(-\tbasez,t).$$ 
$T^{\vu}(-\tbasez) = 0$ because all segments of $(-\tbasez)$ are parallel. 
It remains to show that $T^{\vu}(t, -\tbasez) = T^{\vu}(-\tbasez,t) = 0$, which yields the result.
 
Let $\vw = \vu \times \ez$.
Given $\ell_0 \in t$, we now want to show that $\sum_{\ell \in \tbasez} \tau^{\vu}(\ell,\ell_0) = \sum_{\ell \in \tbasez} \tau^{\vu}(\ell_0,\ell) = 0$. 
This is obvious if $\ell_0$ is not parallel to $\vw$. 
Otherwise, effects cancel out, as illustrated in Figure \ref{fig:example_boxEffects}.
\begin{figure}[ht]%
\centering
\includegraphics[width=0.5\columnwidth]{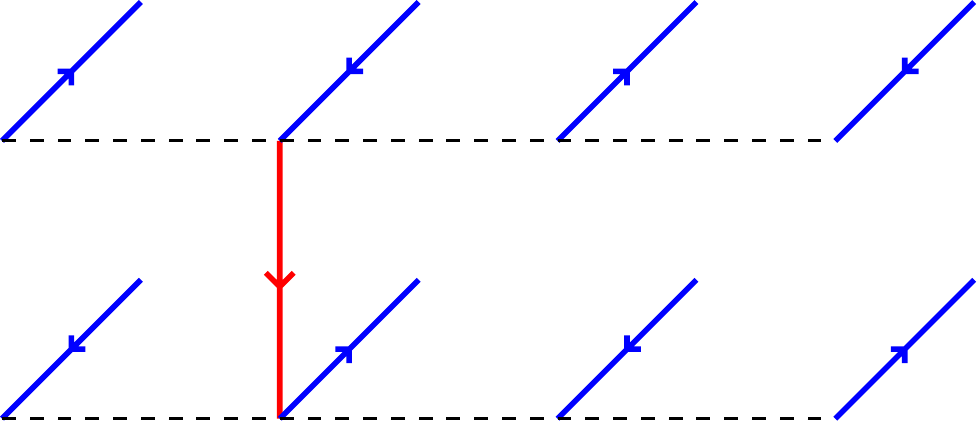}%
\caption{A dimer $\ell_0$ parallel to $\vw$, portrayed in red, and the pairs of segments (blue) of $\tbasez$ affected by it: $\vu$-effects cancel.}%
\label{fig:example_boxEffects}%
\end{figure} 
%
\end{proof}

If $Q \subset \pi$ is a basic square and $\vw \in \Delta$ is a normal vector for $\pi$, define the color of $Q$ to be the same as the color of the basic cube $Q - [0,1]\vw$; and 
$$\ccol(Q) =
\begin{cases}
1, &\text{if } Q \text{ is black;}\\
-1, &\text{if } Q \text{ is white.}
\end{cases}
$$

Recall the definition of $\vu$-shade from Section \ref{sec:combTwistBoxes}. If $A$ is a set of segments or a set of dominoes, $\vu \in \Phi$ and $Q$ is a basic square with normal $\vw \in \Delta$, we set
$$S(A,\vu,Q,n) = \{\ell \in A \,|\, \ell \cap \cS^{\vu}(Q + [0,n]\vw) \neq \emptyset\}.$$

\begin{lemma}
\label{lemma:alternativeFluxFormula}
Let $\cR$ be a $\vw$-cylinder ($\vw \in \Delta$) with base $\cD \subset \pi$ and even depth $N$. Let $Q \subset \pi$ be a basic square, $Q \not\subset \cD$, let $t$ be a tiling of $\cR$ and let $\vu \in \Phi$.
Then
$$\sum_{d \in S(t, \vu,Q,N)} \det(\tv(d), \vw,\vu) = 0.$$  
\end{lemma}
\begin{figure}[ht]%
\centering
\def\svgwidth{0.6\columnwidth}
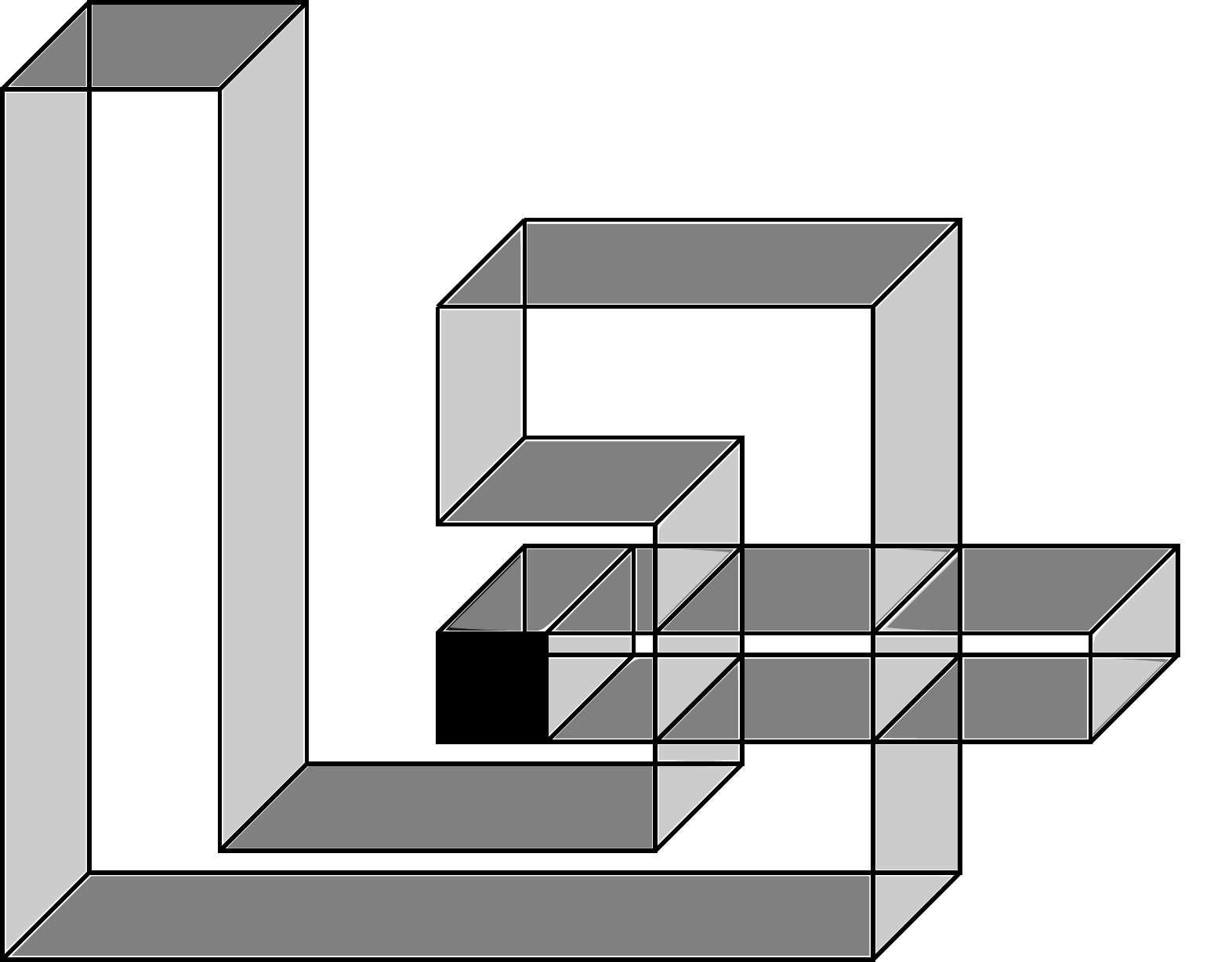
\caption{A cylinder $\cR$ with base $\cD \subset \pi$ and depth $N$, a basic square $Q \subset \pi$, $Q \not\subset \cD$ and the shade $S^{\vu}(Q + [0,N]\vw)$. }%
\label{fig:squareWalls}%
\end{figure}
\begin{proof}
The reader may want to follow by looking at Figure \ref{fig:squareWalls}.
Let $\tbasew = \tbasew(\cR)$, $S_t = S(t,\vu, Q,N)$, and for each $\gamma \in \Gamma^*(t,\tbasew)$, 
let $S_\gamma$ denote $S(\gamma,\vu,Q,N)$.
Clearly,
$$ \sum_{d \in S_t} \det(\tv(d), \vw,\vu) = \sum_{\substack{\gamma \in \Gamma^*(t,\tbasew) \\ \ell \in S_\gamma}} \det(\tv(\ell),\vw,\vu)).$$
Let $p_Q$ be the center of the square $Q$, and let $\Pi$ denote the orthogonal projection on $\pi$. For each $\gamma \in \Gamma^*(t,\tbasew)$, $\Pi \circ \gamma$ is a polygonal curve, so that the winding number of $\gamma$ around $p_Q$ equals (see, e.g., \cite{polygonalWindingNumber} for an algorithmic discussion of winding numbers)
\begin{align*} \wind(\Pi \circ \gamma, p_Q) &= \half\left(\#\{\ell \in S_\gamma \,|\, \tv(\ell) = \vw \times \vu\} - \#\{\ell \in S_\gamma \,|\, \tv(\ell) = -\vw \times \vu\}\right)\\
&= \half \sum_{\ell \in S_\gamma} \det(\tv(\ell), \vw,\vu).
\end{align*}
But $\wind(\Pi \circ\gamma, p_Q) = 0$ ($p_Q \notin \cD$ and $\cD$ is simply connected), so we get the result.
\end{proof}

\begin{prop}
\label{prop:equalTwistsIntNEven}
Let $N \in \NN$ be even, and suppose $\cR$ is a cylinder with depth $N$. If $t$ is a tiling of $\cR$, then $T^{\ex}(t) = T^{\ey}(t) = T^{\ez}(t) \in \ZZ$.
\end{prop}

\begin{proof}
Suppose $\cR = \cD + [0,N]\vw$, where $\cD \subset \pi$ is simply connected and $\vw \in \Delta$ is the axis of the cylinder. 
Let $\cA \subset \pi$ be a rectangle with vertices in $\ZZ^3$ such that $\cD \subset \cA$: this implies that the box
$\cB = \cA + [0,N]\vw \supset \cR$. Let $\vu \in \Phi, \vu \perp \vw$. We want to show that $T^{\vu}(t) = T^{\vw}(t)$.

Let $t$ be a tiling of $\cR$, and let $t_*$ be the tiling of $\cB \setminus \cR$ such that every dimer is parallel to $\vw$. Applying Lemma \ref{lemma:allPretwistsBoxes} to the box $\cB$, we see that $T^{\vu}(t \sqcup t_*) = T^{\vw}(t)$: it remains to show that $T^{\vu}(t \sqcup t_*) - T^{\vu}(t) = 0$.  

Let $\tbasew$ be the tiling of $\cR$ such that every domino is parallel to $\vw$, and let $Q \subset \pi$ be a basic square such that $\interior(Q) \subset \cA \setminus \cD$. Let $t_Q$ be the set of $N/2$ dominoes of $t_*$ contained in $Q + [0,N]\vw$: we have 
$$T^{\vu}(t \sqcup t_*) - T^{\vu}(t) = T^{\vu}(t,t_*) + T^{\vu}(t_*,t) = \sum_{\interior(Q) \subset \cA \setminus \cD} T^{\vu}(t_Q, t) + T^{\vu}(t,t_Q).$$ 
Notice that, for every domino $d \in t_Q$, $\tv(d) = \ccol(Q) \vw$. Moreover, the dominoes in $S_{t,\vu} = S(t,\vu,Q,N)$ are precisely the ones that intersect the $\vu$-shade of at least one domino of $t_Q$, so that
$$T^{\vu}(t_Q, t) = \frac{1}{4}\sum_{d \in S_{t,\vu}} \det(\tv(d),\ccol(Q)\vw,\vu) = \frac{\ccol(Q)}{4}\sum_{d \in S_{t,\vu}} \det(\tv(d),\vw,\vu),$$
which equals $0$ by Lemma \ref{lemma:alternativeFluxFormula}.
Analogously (the first equality below uses Lemma \ref{lemma:twistNegatingDirection}),
$$T^{\vu}(t,t_Q) = T^{-\vu}(t_Q,t) = \frac{\ccol(Q)}{4}\sum_{d \in S(t,-\vu,Q,n)} \det(\tv(d),\vw,-\vu) = 0.$$
Since $T^{\vw}(t) \in \ZZ$ (by Proposition \ref{prop:writheFormula}), we have completed the proof. 
\end{proof}
\begin{lemma}
\label{lemma:equalTwistsNOdd}
Let $N \in \ZZ$ be odd, and let $\cR$ be a cylinder with depth $N$ that admits a tiling $t$. Then $T^{\ex}(t) = T^{\ey}(t) = T^{\ez}(t) \in \half\ZZ$.
\end{lemma}
In fact, we prove in Proposition \ref{prop:equalTwistsIntNOdd} that $T^{\ex}(t) = T^{\ey}(t) = T^{\ez}(t) \in \ZZ$, but for our proof this first step is needed. Also, it is not clear when a cylinder with odd depth $N$ is tileable: see Lemma \ref{lemma:treeConstruction} for a related result.
\begin{proof}
Suppose $\cR$ has base $\cD$ and axis $\vw \in \Delta$, so that $\cR = \cD + [0,N] \vw$, and let $\vu \in \Phi$, $\vu \perp \vw$. 
Let $t$ be a tiling of $\cR$. We want to show that $T^{\vu}(t) = T^{\vw}(t)$.

Consider $\cR' = \cD + [0,2N]\vw$, and the tiling $\hat{t} = t_0 \sqcup t_1$ of $\cR'$ which consists of two copies $t_0$ and $t_1$ of $t$, where $t_0$ tiles the subregion $\cD + [0,N]\vw$ and $t_1$ tiles the subregion $\cD + [N,2N]\vw$.


By Proposition \ref{prop:equalTwistsIntNEven}, $T^{\vu}(\hat{t}) = T^{\vw}(\hat{t}) \in \ZZ$. Now clearly $T^{\vu}(\hat{t}) = 2 T^{\vu}(t)$, because the $\vu$-shades of dimers of $t_0$ do not intersect dimers of $t_1$ (and vice-versa). We need to prove that $T^{\vw}(\hat{t}) = 2 T^{\vw}(t)$. 

Notice that $T^{\vw}(\hat{t}) = T^{\vw}(t_0) + T^{\vw}(t_1) + T^{\vw}(t_0,t_1) = 2T^{\vw}(t) + T^{\vw}(t_0,t_1)$. Let $d_0 \in t_0$, $d_1 \in t_1$ be dominoes, and let $\tilde{d}_0$ and $\tilde{d}_1$ be the dominoes of $t$ that they ``refer to''. If $\tilde{d}_0 \neq \tilde{d}_1$, then clearly
$$d_1 \cap \cS^{\vw}(d_0) \neq \emptyset \Leftrightarrow \tilde{d}_1 \cap (\cS^{\vw}(\tilde{d}_0) \cup \cS^{-\vw}(\tilde{d}_0)) \neq \emptyset $$
and $\tau^{\vw}(d_0,d_1) = \frac{1}{4}\det(\tv(d_1), \tv(d_0), \vw) = - \frac{1}{4}\det(\tv(\tilde{d}_1), \tv(\tilde{d}_0), \vw) = \tau^{-\vw}(\tilde{d}_0,\tilde{d}_1) - \tau^{\vw}(\tilde{d}_0,\tilde{d}_1).$
Therefore,
$T^{\vw}(t_0,t_1) = \sum_{d,d' \in t}\tau^{-\vw}(d,d') - \tau^{\vw}(d,d') = 0$. Consequently, $T^{\vw}(\hat{t}) = 2 T^{\vw}(t)$ and thus $T^{\vu}(t) = T^{\vw}(t)$.

 Moreover, since $T^{\vw}(t) = T^{\vw}(\hat{t})/2$ and $T^{\vw}(\hat{t}) \in \ZZ$, it follows that $T^{\vw}(t) \in \half\ZZ$, which completes the proof.
\end{proof}


\begin{lemma}
\label{lemma:integerTwistDifference}
Let $\cD \subset \pi$ be a planar region, and let $\vw \in \Delta$ be the normal vector for $\pi$. For each $k \in \NN$, write $\cR_k = \cD + [0,2k+1]\vw$. If $k_1, k_2 \in \NN$, then for each $\vu \in \Phi$ and every pair of tilings $t_1$ of $\cR_{k_1}$, $t_2$ of $\cR_{k_2}$, $T^{\vu}(t_1) - T^{\vu}(t_2) \in \ZZ$.
\end{lemma}
Should $\cR_{k_1}$ or $\cR_{k_2}$ not be tileable, the statement is vacuously true.
\begin{proof}
By Lemma \ref{lemma:equalTwistsNOdd}, it suffices to show the result for $\vu \perp \vw$.
Let $t_1$ and $t_2$ be tilings of $\cR_{k_1}$ and $\cR_{k_2}$, respectively. Consider the cylinder with even depth $\cR = \cD + [0,2k_1 + 2k_2 + 2]\vw$, and let $\tilde{t}_2$ denote the tiling of $\cD + [2k_1 + 1,2k_1 + 1 + 2k_2 + 1]\vw$ which is a copy of $t_2$. If $t = t_1 \sqcup \tilde{t}_2$, then $T^{\vu}(t) \in \ZZ$, by Proposition \ref{prop:equalTwistsIntNEven}. Also, since $\vu \perp \vw$, $T^{\vu}(t) = T^{\vu}(t_1) + T^{\vu}(\tilde{t}_2) = T^{\vu}(t_1) + T^{\vu}(t_2)$, so that  
$$ T^{\vu}(t_1) - T^{\vu}(t_2) = T^{\vu}(t_1) + T^{\vu}(t_2) - 2T^{\vu}(t_2) = T^{\vu}(t) - 2T^{\vu}(t_2).$$ 
Since, by Lemma \ref{lemma:equalTwistsNOdd}, $2T^{\vu}(t_2) \in \ZZ$, we're done. 
%
%
\end{proof}

\begin{lemma}
\label{lemma:treeConstruction}
Let $\pi$ be a basic plane with normal $\vw \in \Delta$, and let $\cD \subset \pi$ be a planar region with connected interior such that 
$$\# (\text{black squares in } \cD) = \# (\text{white squares in } \cD) = n.$$ 
Then there exists a tiling $t_0$ of $\cD +[0,2n-1]\vw$ such that $T^{\vw}(t_0) \in \ZZ$.
\end{lemma}

Notice that, with Lemma \ref{lemma:treeConstruction}, the proof of Proposition \ref{prop:equalTwistsMultiplex} is complete. However, we need some preparation before we can prove Lemma \ref{lemma:treeConstruction}.


It is a well-known fact that domino tilings of a region can be seen as perfect matchings of a related graph: in fact, if we consider the graph whose vertices are centers of the cubes (squares in the planar case) of the region, and where two vertices are joined if their Euclidean distance is $1$, then a domino tiling can be directly translated as a perfect matching in this graph. This graph is called the \emph{associated graph} of a region $\cR$ (planar or spatial), and denoted $G(\cR)$.
Since the proof of Lemma \ref{lemma:treeConstruction} will come more naturally in the setting of matchings in associated graphs, we shall revert to this viewpoint for what follows. 

A \emph{bicoloring} of a graph $G$ is a coloring of each vertex of $G$ as black or white, in such a way that no two adjacent vertices have the same color.
Associated graphs for a region $\cR$ are always bicolored: each vertex inherits the color of the cube (or square) it refers to.
For what follows, we shall assume that all graphs are already bicolored. 
Moreover, any subgraph of a bicolored graph $G$ (for instance, the one obtained after deleting a vertex) shall inherit the bicoloring of $G$.  

\begin{lemma}
\label{lemma:leafColors}
Let $T$ be a bicolored tree. If all leaves are white, then the number of white vertices in $T$ is strictly larger than the number of black vertices in $T$.
\end{lemma}
By definition, a tree is connected and, therefore, nonempty.
\begin{proof}
We proceed by induction on the number of vertices. The result is clearly true if $T$ has three or fewer vertices.
Suppose, by induction, that the result holds for balanced trees with $m$ vertices for any $m < n$. Let $T$ be a tree with $n$ vertices such that all leaves are white.

Let $w \in T$ be a (white) leaf, and let $v \in T$ be the only neighbor of $w$. Let $F$ be the forest obtained by deleting $w$ and $v$: $F$ is nonempty, otherwise $v$ would have to be a black leaf, which contradicts the hypothesis. Now for each connected component $T'$ of $F$, $T'$ is a tree with less than $n$ vertices such that all leaves are white: therefore, by induction, $F$ has more white vertices than black vertices. However, the vertices of $T$ are those of $F$ plus one black vertex ($v$) and one white vertex ($w$), so that the number of white vertices in $T$ is greater than that of black ones.
By induction, we get the result.  
\end{proof}

A connected bicolored graph $G$ is \emph{balanced} if the number of white vertices equals the number of black ones. By Lemma \ref{lemma:leafColors}, a balanced tree must have at least one white leaf and one black leaf.

A \emph{perfect matching} of a bipartite graph $G$ is a set of pairwise disjoint edges of $G$, such that every vertex is adjacent to (exactly) one of the edges in the matching. Clearly, a necessary condition for the existence of a perfect matching is that $G$ is balanced.

Let $G = (V,E)$ be a bicolored graph (in this notation, $V$ is the vertex set of $G$, and $E$ is its edge set), and let $I_n = \{0,1, \ldots, n-1\}$. Let $G \times I_n = (V \times I_n,E_n)$, where $E_n$ consists of all edges connecting $(v,j)$ and $(v,j+1)$, for each $v \in V$ and $j \in I_{n-1}$, plus the edges connecting $(v_1,j)$ and $(v_2,j)$ for each $j \in I_n$ whenever the edge $v_1v_2 \in E$. The color of a vertex $(v,j) \in G \times I_n$ equals the color of $v$ if and only if $j$ is even. Naturally, if $\cD \subset \pi$ is a planar region with normal $\vw$, then $G(\cD) \times I_n \approx G(\cD + [0,n]\vw)$.

Let $G$ be a (nonempty) balanced connected bicolored graph with $2n$ vertices. Algorithm \ref{algorithm:perfectMatching} finds a perfect matching $M$ of $G \times I_{2n-1}$.
\begin{algorithm}
\caption{Algorithm for finding a perfect matching $M$ of $G \times I_{2n-1}$.}
\label{algorithm:perfectMatching}
\begin{algorithmic}
\State Pick a spanning tree $T$ for $G$.
\Comment \textit{\footnotesize $T$ is a balanced tree}
\State $M_0 \gets \emptyset$
\State $T_0 \gets T$
\State $k \gets 0$
\While {$T_k \neq \emptyset$}
\State Pick a white leaf $v_w$ and a black leaf $v_b$ of $T_k$ \newline
\Comment \textit{ \footnotesize Lemma \ref{lemma:leafColors} ensures that a balanced tree has at least one white leaf and one black leaf}
\State Pick a path $P_k = v_{k,1}v_{k,2} \ldots v_{k,m_k}$ in $T_k$ from $v_w$ to $v_b$ \newline
\Comment{\textit{ \footnotesize i.e., $v_{k,1} = v_w, v_{k,m_k} = v_b$; notice that $m_k$ is necessarily even}}
\State $D_k \gets \{(v,2k-1)(v,2k) \,|\, v \in T \setminus T_k\} \sqcup \{(v,2k)(v,2k+1)\,|\, v \in T_k \setminus P_k)\}$
\State $E_k \gets \{(v_{k,2i-1}v_{k,2i},2k) \,|\, 1 \leq i \leq \frac{m_k}{2}\} \sqcup \{(v_{k,2i}v_{k,2i+1},2k+1) \,|\, 1 \leq i < \frac{m_k}{2}\}$\newline        
\Comment{\textit{\footnotesize Here $(vw,l)$ means the edge $(v,l)(w,l)$, i.e., the edge between the vertices $(v,l)$ and $(w,l)$}}
\State $M_{k+1} \gets M_k \sqcup D_k \sqcup E_k$
\State $T_{k+1} \gets T_k \setminus \{v_w, v_b\}$ \newline
\Comment{\textit{\footnotesize Notice that $T_{k+1}$ is still a balanced tree (except in the last iteration, when it is empty)}}
\State $k \gets k + 1$
\EndWhile 
\State $M \gets M_k$
\end{algorithmic}
\end{algorithm}

\begin{lemma}
If $G$ is a connected bicolored balanced graph with $2n$ vertices, then the set of edges $M$ generated by running Algorithm \ref{algorithm:perfectMatching} on $G$ is a perfect matching of $G \times I_{2n-1}$.
\end{lemma}
\begin{proof}
To see that $M \subset E(G \times I_{2n-1})$, notice that any spanning tree has $2n$ vertices, and exactly two vertices are deleted in each iteration, so that the last iteration where $T_k \neq \emptyset$ occurs when $k = n-1$. In all other iterations (i.e., $0 \leq k < n-1$), clearly any edge created is contained in $E(G \times I_{2n-1})$. When $k = n-1$, $T_k$ is a balanced tree with two vertices, so $P_k = v_w v_b$ and only the edges $\{(v,2n-3)(v,2n-2) \,|\, v \in T \setminus T_k\}$ plus the edge $(v_w,2n-2)(v_b,2n-2)$ are created. Now these edges are contained in $E(G \times I_{2n-1})$, so we're done.  

This proves that $M$ is a subset of the edgeset of $G \times I_{2n-1}$. The reader will easily convince himself that it is a perfect matching (i.e., that every vertex $(v,j)$ of $G \times I_{2n-1}$ is adjacent to exactly one edge of $M$).
\end{proof}

Next we shall prove Lemma \ref{lemma:treeConstruction}. In order to make the explanation clearer, we shall first introduce a few concepts. Let $G$ be a bicolored connected balanced graph, and consider the perfect matching $M$ of $G \times I_{2n-1}$ obtained by running Algorithm \ref{algorithm:perfectMatching} on $G$, as well as the intermediate objects that were created, such as $E_k$ and $P_k$.

Given an edge $e = (vw,j)$ of $E_k$, we say $e$ is adjacent to $v$ and to $w$ (even though it is not an edge of $G$). For $v \in G$ and $0 \leq k \leq 2n-1$, we write $E(v,k) = \{e \in E_k \,|\, e \text{ adjacent to } v\}$. 

Consider the paths $P_k = v_{k,1}v_{k,2} \ldots v_{k,m_k}$ chosen in each step of the algorithm (we shall also use this notation in the proof).
If $j > k$, we say that a path $P_j$ \emph{meets} $P_k$ \emph{at} $v \in G$ if $v = v_{j,i} \in P_k$ for some $i > 1$, but $v_{j,i-1} \notin P_k$. Analogously, $P_j$ \emph{leaves} $P_k$ \emph{at} $v$ if $v = v_{j,i} \in P_k$ for some $i < m_j$, but $v_{j,i+1} \notin P_k$. Notice that a path $P_j$ can meet and leave $P_k$ at the same vertex $v$. Also, notice that $P_j$ can only meet (resp. leave) $P_k$ at most once (i.e., at no more than one vertex).



\begin{proof}[Proof of Lemma \ref{lemma:treeConstruction}]
Consider the graph $G = G(\cD)$ associated with the planar region $\cD$. Clearly $G$ is balanced; since $\cD$ has connected interior, it follows that $G$ is also connected. Let $M$ be the perfect matching obtained after running Algorithm \ref{algorithm:perfectMatching} on $G$, and let $t$ be the tiling of $\cD + [0,2n-1]\vw$ associated with $M$.

If $e_0, e_1 \in M$, we will abuse notation and write $\tau^{\vw}(e_0,e_1) = \tau^{\vw}(d_0,d_1)$, where $d_i \in t$ is the domino associated with $e_i \in M$: we also say that two edges are parallel if their associated dominoes are parallel. 

Notice that the only dominoes that are not parallel to $\vw$ are those associated with the edges of $E_k$ for each $k$: therefore, 
$$
T^{\vw}(t) = \sum_{i \leq j} T^{\vw}(E_i, E_j) = \sum_{\substack{i \leq j \\e \in E_i, e' \in E_j}}\tau^{\vw}(e,e').
$$ 
Fix $0 \leq k \leq n-1$. We want to show that $\sum_{j \geq k} T^{\vw}(E_k, E_j) \in \ZZ$. First, write 
$$\sum_{j \geq k} T^{\vw}(E_{k}, E_j) = \sum_{\substack{1 < i < m_{k}\\j \geq k}}T^{\vw}(E(v_{k,i},k), E(v_{k,i},j));$$
we may ignore $v_{k,1}$ and $v_{k,m_k}$ because they are deleted from the tree in step $k$, so that $E(v_{k,1},j) = E(v_{k,m_k},j) = \emptyset$ for each $j > k$ (for $j = k$, it contains only one edge, so there is also no effect).

\begin{figure}[ht]%
\centering
\subfloat[Some cases where $P_k$ goes straight at $v$: the effects are, respectively, $1$, $1/2$, $1/2$ and $0$.]{\includegraphics[width=0.45\columnwidth]{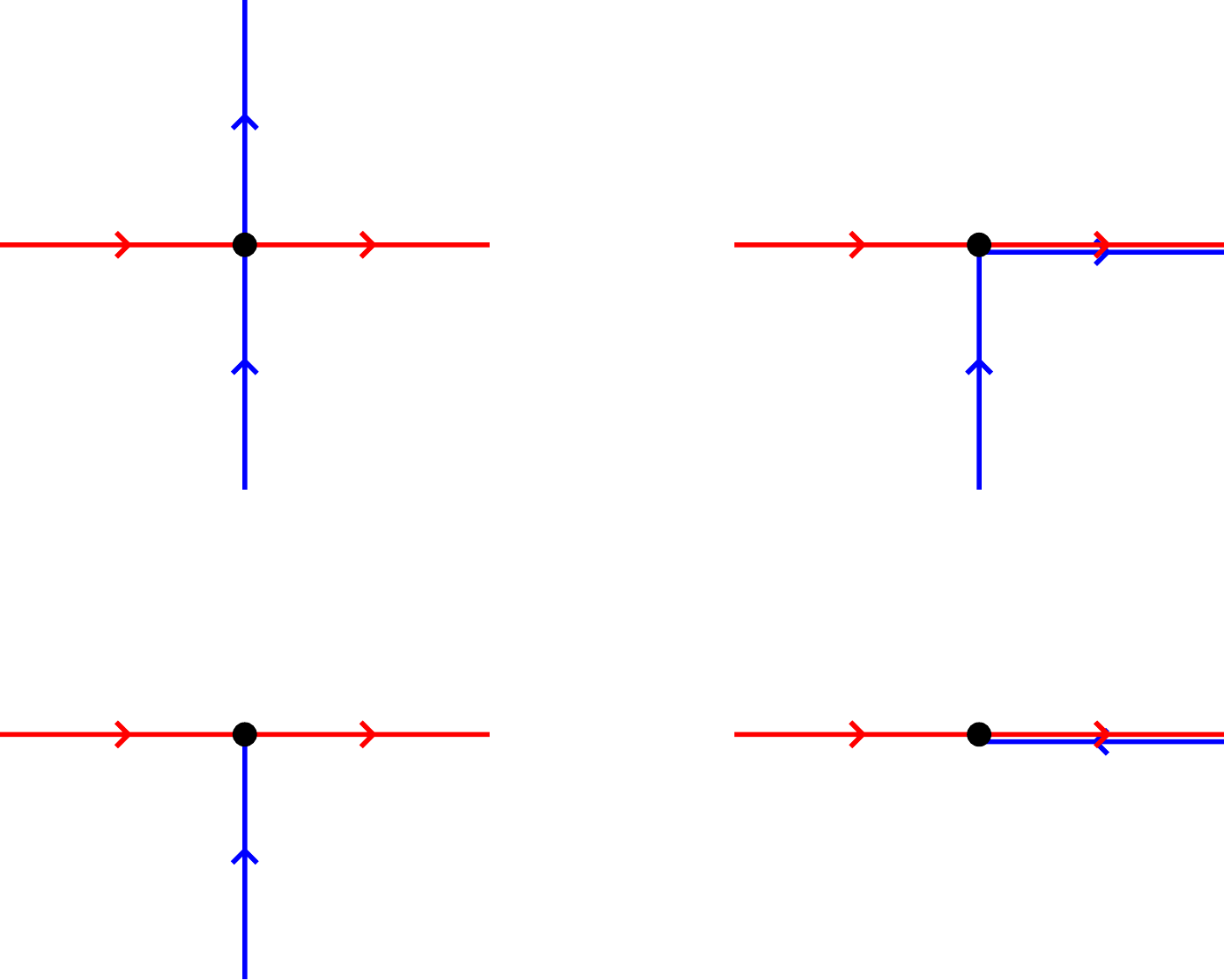}\label{subfig:treeConstructionStraight}} 
\hspace{0.075\columnwidth}
\subfloat[Some cases where $P_k$ makes a left turn at $v$: in this case the red segments have nonzero effect on one another (in this case it is $1/4$): the effects on the blue segments are, respectively, $1/2$, $0$, $1/4$ and $-1/4$.]{\includegraphics[width=0.45\columnwidth]{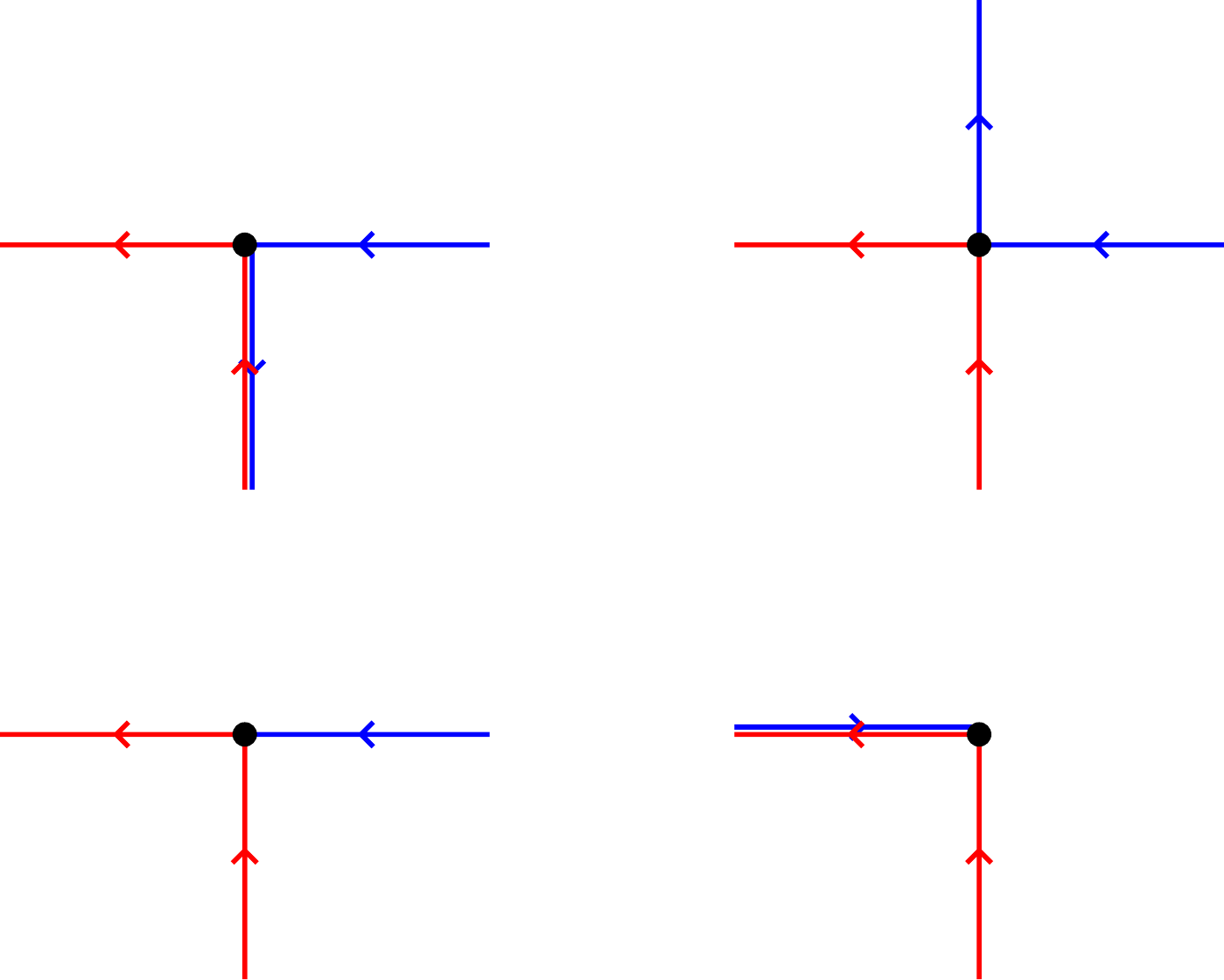}\label{subfig:treeConstructionTurn}}
\caption{Edges of $E_k$ (red) and $E_j$ (blue) for some $j > k$, portrayed as edges of $G$. The edges are oriented as $\tv(d)$, where $d$ is the associated domino. The portrayed vertex is $v$, which we assume here to be black: notice that $v$ is one of the endpoints of $P_j$ in the bottom two cases of each figure.}%
\label{fig:treeConstructionEffects}%
\end{figure}

For $v = v_{k,i}$, $1 < i < m_k$, we claim that, modulo $1$, 
$\sum_{j \geq k} T^{\vw}(E(v,k), E(v,j))$ equals  
$$\half \left(\#\{j > k \,|\, P_j \text{ meets } P_k \text{ at } v\} + \#\{j > k \,|\, P_j \text{ leaves } P_k \text{ at } v\}\right)$$ 
(in other words, their difference is an integer).

If the two edges in $E(v,k)$ are parallel (i.e., $P_k$ goes straight at $v$), then $T^{\vw}(E(v,k), E(v,k)) = 0$. By checking a number of cases (see Figure \ref{subfig:treeConstructionStraight}) we see that the following holds for each $j > k$: 
\begin{equation}
T^{\vw}(E(v,k), E(v,j)) = 
\begin{cases}
\pm 1, &P_j \text{ meets and leaves } P_k \text{ at } v; \\
\pm 1/2, &P_j \text{ either meets or leaves } P_k \text{ at } v; \\ 
0, &\text{otherwise.}
\end{cases}
\label{eq:effectsMeetingAndLeaving}
\end{equation}

If the two edges in $E(v,k)$ are not parallel (i.e., $P_k$ makes a turn at $v$), we proceed as follows: assume that the path $P_k$ makes a left turn and that $v$ is a black vertex (the other cases are analogous). Let $k'$ be the step where $v$ is chosen as the black leaf to be deleted (so that $v = v_{k', m_{k'}}$): again, inspection of a few possible cases (some of which are shown in Figure \ref{subfig:treeConstructionTurn}) shows that \eqref{eq:effectsMeetingAndLeaving} holds for $k < j < k'$ (and for $j > k'$, obviously $T^{\vw}(E(v,k), E(v,j)) = 0$).
Also, $T^{\vw}(E(v,k), E(v,k)) = 1/4$ (because it is a left turn and $v$ is black), and (see the last two examples in Figure \ref{subfig:treeConstructionTurn})
$$
T^{\vw}((E(v,k), E(v,k')) = 
\begin{cases}
1/4, &P_{k'} \text{ meets } P_k \text{ at } v;\\
-1/4, &\text{otherwise;}
\end{cases}
$$
so that $T^{\vw}(E(v,k), E(v,k)) + T^{\vw}((E(v,k), E(v,k')) = 1/2$ if and only if $P_{k'}$ meets $P_j$ at $v$ (and $0$ otherwise), so that we get the result.

Now let $N(v) = \#\{j > k \,|\, P_j \text{ meets } P_k \text{ at } v\} + \#\{j > k \,|\, P_j \text{ leaves } P_k \text{ at } v\}.$ To finish the proof, we need to show that 
$$
N = \sum_{1 < i < m_k}N(v_{k,i}) = \#\{j > k \,|\, P_j \text{ meets } P_k\} + \#\{j > k \,|\, P_j \text{ leaves } P_k\}
$$ is even.
Because all $P_j$'s are paths in a tree $T_k$, it follows that each path meets (or leaves) $P_k$ at most once. Therefore, each $j > k$ may contribute $0$ (if it never meets nor leaves $P_k$), $1$ (if it either meets or leaves $P_k$, but not both) or $2$ (if it meets and leaves $P_k$) to the above sum. This contribution is $0$ if $v_{j,1}, v_{j,m_j} \in P_k$; it is $0$ or $2$ if $v_{j,1}, v_{j,m_j} \notin P_k$. If exactly one of the two is in $P_k$, the contribution is $1$; however, since
$\#\{j > k \,|\, v_{j,1} \in P_k, v_{j,m_j} \notin P_k\} = \#\{j > k \,|\, v_{j,1} \notin P_k, v_{j,m_j} \in P_k\},$  
it follows that $N$ is even, so that
$T^{\vw}(t) = \sum_{j \geq k} T^{\vw}(E_{k}, E_j) \equiv N/2 \pmod{1}$
is an integer. 
\end{proof}
We sum up our main results in the following proposition:
\begin{prop}
\label{prop:equalTwistsIntNOdd}
Let $\cD \subset \pi$ be a planar region with normal vector $\vw$ and connected interior such that
$$\# (\text{black squares in } \cD) = \# (\text{white squares in } \cD) = n.$$ 
Then $\cD + [0,2n-1]\vw$ is tileable.
Moreover, for each $k \in \NN$ such that $\cD + [0,2k-1]\vw$ is tileable (in particular, for each $k \geq n$), every tiling $t$ of $\cD + [0,2k-1]\vw$ satisfies $T^{\ex}(t) = T^{\ey}(t) = T^{\ez}(t) \in \ZZ$.
\end{prop}
\begin{proof}
Follows directly from Lemmas \ref{lemma:equalTwistsNOdd}, \ref{lemma:integerTwistDifference} and \ref{lemma:treeConstruction}.
\end{proof}

Now that we have seen that the twist, as in Definition \ref{def:twist}, is well-defined for cylinders, we may adopt the notation $\Tw(t)$ when $t$ is a tiling of a cylinder. 
\section{Additive properties and proof of Theorem \ref{theo:main}}
\label{sec:additiveProperties}

The goal for this section is to discuss some additive properties of the twist and to complete the proof of Theorem \ref{theo:main}. 


\begin{lemma}
\label{lemma:secondOrderTwistDiff}
Let $\cR_0$ and $\cR_1$ be two regions whose interiors are disjoint. Let $t_{\cR_0,0}$ and $t_{\cR_0,1}$ be two tilings of $\cR_0$ and $t_{\cR_1,0}$ and $t_{\cR_1,1}$ be two tilings of $\cR_1$. For each $(i,j) \in \{0,1\}^2$, set $t_{ij} = t_{\cR_0,i} \sqcup t_{\cR_1,j}$, which is a tiling of $\cR = \cR_0 \cup \cR_1$. Let $\Gamma_i^* = \Gamma^*(t_{\cR_i,0}, t_{\cR_i,1})$, $i=0,1$. Then, for each $\vu \in \Phi$,
$$
T^{\vu}(t_{00}) - T^{\vu}(t_{01}) - T^{\vu}(t_{10}) + T^{\vu}(t_{11}) = 2\sum_{\gamma_0 \in \Gamma_0^*,\gamma_1 \in \Gamma_1^*} \Link(\gamma_0,\gamma_1). 
$$
In particular, if $\cR_0$ or $\cR_1$ is simply connected, then $T^{\vu}(t_{00}) - T^{\vu}(t_{01}) - T^{\vu}(t_{10}) + T^{\vu}(t_{11}) = 0$.
\end{lemma}
\begin{proof}
For shortness, given two sets of segments $A_0$ and $A_1$, we shall in this proof write $\ST^{\vu}(A_0,A_1) = T^{\vu}(A_0,A_1) + T^{\vu}(A_1,A_0)$.

For each $(i,j) \in \{0,1\}^2,$ we have 
$$
T^{\vu}(t_{ij}) = T^{\vu}(t_{\cR_0,i} \sqcup t_{\cR_1,j})
= T^{\vu}(t_{\cR_0,i}) + T^{\vu}(t_{\cR_1,j}) + \ST^{\vu}(t_{\cR_0,i},t_{\cR_1,j}).
$$
Notice that the last term is the only one that depends on both $i$ and $j$, so that it is the only one that does not cancel out in the sum $\sum_{i,j \in \{0,1\}} (-1)^{i+j} \Tw(t_{ij})$.
Therefore, we have
\begin{align*}
&\sum_{i,j \in \{0,1\}} (-1)^{i+j} T^{\vu}(t_{ij}) = \sum_{i,j \in \{0,1\}} (-1)^{i+j} \ST^{\vu}(t_{\cR_0,i},t_{\cR_1,j})=\\
&= \ST^{\vu}(t_{\cR_0,0} \sqcup (-t_{\cR_0,1}), t_{\cR_1,0} \sqcup (-t_{\cR_1,1})) 
= \sum_{\gamma_0 \in \Gamma_0^*, \gamma_1 \in \Gamma_1^*} \ST^{\vu}(\gamma_0,\gamma_1).
\end{align*}
%
Since for each pair $\gamma_0, \gamma_1$ in the sum we have $\gamma_i \subset \interior(\cR_i)$, it follows that $\gamma_0 \cap \gamma_1 = \emptyset$.
Hence, by Lemma \ref{lemma:linkingNumber}, $\ST^{\vu}(\gamma_0,\gamma_1) = 2 \Link(\gamma_0, \gamma_1)$, which yields the result. 
\end{proof}

\begin{coro}
\label{coro:embeddableRegions}
Let $\cR$ be a simply connected region, and suppose that there exists a box $\cB \supset \cR$ such that $\cB \setminus \cR$ is tileable. If $t_0, t_1$ are two tilings of $\cR$ and $t_a, t_b$ are two tilings of $\cB \setminus \cR$, then
$$ 
\Tw(t_0 \sqcup t_a) - \Tw(t_1 \sqcup t_a) = \Tw(t_0 \sqcup t_b) - \Tw(t_1 \sqcup t_b).
$$ 
\end{coro}
\begin{proof}
Use Lemma \ref{lemma:secondOrderTwistDiff} with $\cR_0 = \cR$, $\cR_1 = \cB \setminus \cR$.
\end{proof}

\begin{lemma}
\label{lemma:trivialTilingBoxes}
Suppose $L,M,N$ are even positive integers, and let $\cB = [0,L] \times [0,M] \times [0,N]$. If $\cR \subset \cB$ is a cylinder with even depth, then there exists a tiling $t_*$ of $\cB \setminus \cR$ such that 
$\Tw(t \sqcup t_*) = \Tw(t)$
for each tiling $t$ of $\cR$.
\end{lemma}
Corollary \ref{coro:embeddableRegions} and Lemma \ref{lemma:trivialTilingBoxes} imply that for any tiling $\tilde{t}_*$ of $\cB \setminus \cR$, there exists a constant $K$ such that, for any tiling $t$ of $\cR$, $\Tw(t \sqcup \tilde{t}_*) = \Tw(t) + K$.  
\begin{proof}
We may without loss of generality assume that the axis of $\cR$ is $\ez$, so that $\cR = \cD + [E,F]\ez$, where $\cD \subset [0,L] \times [0,M] \times \{0\}$ and $F - E$ is even.   

Let $\cB_1 = [0,L] \times [0,M] \times [E,F]$. Clearly there exists a tiling $t_{1,*}$ of $\cB_1 \setminus \cR$ such that every domino is parallel to $\ez$: hence, $\Tw(t \sqcup t_{1,*}) = T^{\ez}(t \sqcup t_{1,*}) = T^{\ez}(t) = \Tw(t)$ for each tiling $t$ of $\cR$. On the other hand, since $L$ is even, there exists a tiling $t_{2,*}$ of $\cB \setminus \cB_1$ such that every dimer is parallel to $\ex$, so that $\Tw(t \sqcup t_{2,*}) = T^{\ex}(t \sqcup t_{2,*}) = T^{\ex}(t) = \Tw(t)$ for each tiling $t$ of $\cB_1$. Setting $t_* = t_{1,*} \sqcup t_{2,*}$ we get the result. 
\end{proof}
\begin{lemma}
\label{lemma:trivialTilingAllMultiplexes}
Let $\cR$ be a tileable cylinder with base $\cD$, axis $\vw \in \Delta$ and depth $n$. Let $\cR' = \cD + [0,2n]\vw$ be a cylinder with even depth formed by two copies of $\cR$; let $\cB \supset \cR'$ be a box with all dimensions even. Then there exist a tiling $t_*$ of $\cB \setminus \cR$ and a constant $K$ such that, for each tiling $t$ of $\cR$, $\Tw(t \sqcup t_*) = \Tw(t) + K.$
\end{lemma}
\begin{proof}
By Lemma \ref{lemma:trivialTilingBoxes}, there exists a tiling $\tilde{t}$ of $\cB \setminus \cR'$ such that $\Tw(t \sqcup \tilde{t}) = \Tw(t)$ for each tiling $t$ of $\cR'$.
Fix a tiling $t_0$ of $\cD + [n,2n]\vw$ (which is tileable because $\cR$ is tileable). If we set $t_* = t_0 \sqcup \tilde{t}$ and $K = \Tw(t_0)$, then for every tiling $t$ of $\cR$,
$$ \Tw(t \sqcup t_*) = \Tw(t \sqcup t_0 \sqcup \tilde{t}) = \Tw(t \sqcup t_0) = \Tw(t) + \Tw(t_0);$$
the last equality holding by fixing $\vu \in \Phi$, $\vu \perp \vw$ and writing $\Tw(t \sqcup t_0) = T^{\vu}(t \sqcup t_0) = T^{\vu}(t) + T^{\vu}(t_0)$. 
\end{proof}

\begin{proof}[Proof of Theorem \ref{theo:main}]
The twist is constructed in Definition \ref{def:twist} and its integrality follows from Proposition \ref{prop:equalTwistsMultiplex}. Lemma \ref{lemma:fullyBalancedMultiplex} and Proposition \ref{prop:flipsAndTrits} yield items \ref{item:flips} and \ref{item:trits} . 
To see item \ref{item:duplexes}, let $\cR$ be a duplex region with axis $\vw$, and consider the tiling $t_{\vw}$ such that all dominoes are parallel to $\vw$: clearly $\Tw(t_{\vw}) = P_{t_{\vw}}'(1) = 0$ (we assume that the reader is familiar with the notation from \cite{primeiroartigo}). Since the space of domino tilings of $\cR$ is connected by flips and trits (\cite[Theorem 2]{primeiroartigo}), Proposition \ref{prop:flipsAndTrits}, together with Theorems 1 and 2 from \cite{primeiroartigo}, implies that for each tiling $t$ of $\cR$, $\Tw(t) = P_t'(1)$ (for a more direct proof of item \ref{item:duplexes}, see \cite{webpageNicolau}).

We're left with proving item \ref{item:multiplexUnion}.
Let $\cR$ be a cylinder, and suppose $\cR = \bigcup_{1 \leq j \leq m} \cR_j$, where each $\cR_j$ is a cylinder (they need not have the same axis) and $\interior(\cR_i) \cap \interior(\cR_j) = \emptyset$ if $i \neq j$. Suppose the bases, axes and depths are respectively, $\cD,\vw,n$ and $\cD_j, \vw_j, n_j$.

Let $t_{j,0}$ and $t_{j,1}$ be two tilings of $\cR_j$. It suffices to show that
$$\Tw\left(\bigsqcup_{1 \leq j \leq m} t_{j,1}\right) - \Tw\left(\bigsqcup_{1 \leq j \leq m} t_{j,0}\right) = \sum_{1 \leq j \leq m} (\Tw(t_{j,1}) - \Tw(t_{j,0})).$$ 

For $0 \leq j \leq m$, let $t_j = \bigsqcup_{1 \leq i \leq j} t_{i,1} \sqcup \bigsqcup_{j < i \leq m} t_{i,0}$. We want to show that $\Tw(t_m) - \Tw(t_0) = \sum_{1 \leq j \leq m} (\Tw(t_{j,1}) - \Tw(t_{j,0}))$.

Let $\cB$ be a box with all dimensions even such that $\cD + [0,2n]\vw \subset \cB$ and $\cD_j + [0,2n_j]\vw_j \subset \cB$ for $j=1,\ldots,m$. By Lemma \ref{lemma:trivialTilingAllMultiplexes}, there exist: a tiling $t_*$ of $\cB \setminus \cR$ and a constant $K$; and for each $j$, a tiling $t_{j,*}$ of $\cB \setminus \cR_j$ and a constant $K_j$ such that $\Tw(t \sqcup t_{*}) = \Tw(t) + K$ for each tiling $t$ of $\cR$, and $\Tw(t \sqcup t_{j,*}) = \Tw(t) + K_j$ for each tiling $t$ of $\cR_j$.

Write $\hat{t}_j = t_* \sqcup \bigsqcup_{1 \leq i < j} t_{i,1} \sqcup \bigsqcup_{j < i \leq m} t_{i,0}$ for each $j$, so that $\hat{t}_j$ is a tiling of $\cB \setminus \cR_j$. Notice that, for $1 \leq j \leq m$, 
$t_j \sqcup t_* = t_{j,1} \sqcup \hat{t}_j$ and $t_{j-1} \sqcup t_* = t_{j,0} \sqcup \hat{t}_j.$  
Therefore, we have
\begin{align*}
&\Tw(t_m) - \Tw(t_0)
= \sum_{1 \leq j \leq m} (\Tw(t_j) - \Tw(t_{j-1}))\\ 
&= \sum_{1 \leq j \leq m} ((\Tw(t_j \sqcup t_*) - K) - (\Tw(t_{j-1} \sqcup t_*) - K))\\
&= \sum_{1 \leq j \leq m} (\Tw(t_{j,1} \sqcup \hat{t}_j) - \Tw(t_{j,0} \sqcup \hat{t}_j)) 
\stackrel{\dagger}{=} \sum_{1 \leq j \leq m} (\Tw(t_{j,1} \sqcup t_{j,*}) - \Tw(t_{j,0} \sqcup t_{j,*}))\\
&= \sum_{1 \leq j \leq m} ((\Tw(t_{j,1}) + K_j) - (\Tw(t_{j,0}) + K_j)) = \sum_{1 \leq j \leq m} (\Tw(t_{j,1}) - \Tw(t_{j,0})).
\end{align*}
Equality $\dagger$ holds by 
Corollary \ref{coro:embeddableRegions}, because $\hat{t}_j$ and $t_{j,*}$ are two tilings of $\cB \setminus \cR_j$.
%
\end{proof}


\section{Examples and counterexamples}
\label{sec:examples}
In this short section, we give a few examples and counterexamples that help motivate the theory and some of the results obtained.

For instance, when looking at Proposition \ref{prop:equalTwistsMultiplex}, one might wonder whether the pretwists are always integers or if they always coincide, at least for, say, simply connected or contractible regions. This turns out not to be the case, as Figure \ref{subfig:twoDominoes} shows: for the tiling $t$ portrayed there, $T^{\ex}(t) = T^{\ey}(t) = 0$ but $T^{\ez}(t) = 1/4$.

\begin{figure}%
\centering
\subfloat[A tiling $t$ satisfying $T^{\ez}(t) = 1/4$.]{\includegraphics[width=0.2\columnwidth]{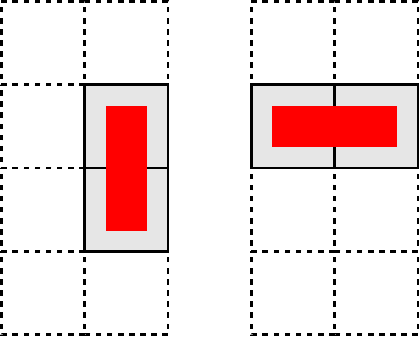}\label{subfig:twoDominoes}}%
\qquad \qquad
\subfloat[A tiling of a pseudocylinder where $T^{\ex}(t) \neq T^{\ez}(t)$.]{\includegraphics[width=0.3\columnwidth]{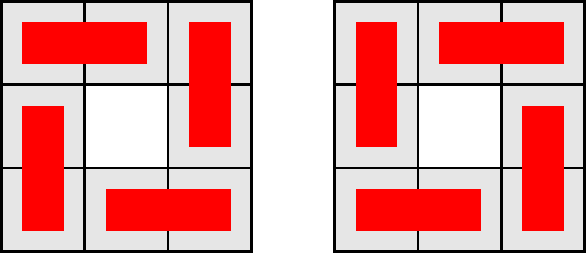}\label{subfig:pseudomultiplex}}
\caption{Two examples of tilings: $\ez$ is chosen to point towards the paper.}%
\label{fig:examples}%
\end{figure}

One might ask whether the pretwists coincide in a pseudocylinder (i.e., if the base is not necessarily simply connected): the tiling $t$ portrayed in Figure \ref{subfig:pseudomultiplex} satisfies $T^{\ex}(t) = T^{\ey}(t) = 0$ and $T^{\ez}(t) = 1$. 
One can prove that they coincide if the pseudocylinder has odd depth (via a modification in the proofs of Proposition \ref{prop:equalTwistsIntNEven} and \ref{prop:equalTwistsIntNOdd}), but we shall not dwell on this. 

\begin{figure}%
\centering
\subfloat[A tiling of a cylinder with depth $3$ allowing no flips or trits.]{\includegraphics[width=0.5\columnwidth]{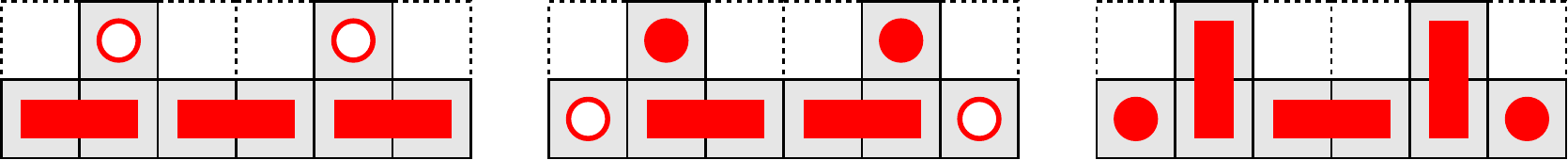}}%
\\
\subfloat[A tiling of a cylinder of depth $6$ whose flip and trit connected component contains only one other tiling.]{\includegraphics[width=\columnwidth]{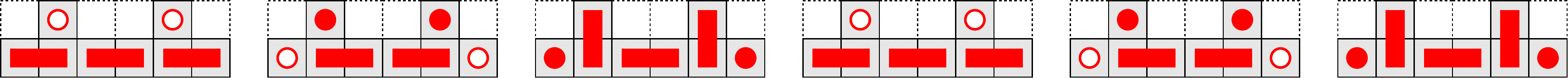}}%
\caption{Examples of regions whose space of tilings is not connected by flips and trits. Notice that these regions don't allow room for a trit, and flips are clearly insufficient to connect the space.}%
\label{fig:notConnectedByLocalMoves}%
\end{figure}

Connectivity by flips and trits does not hold for cylinders in general, unlike the case with two floors. Figure \ref{fig:notConnectedByLocalMoves} shows two counter-examples, one with odd depth and one with even depth. We do not know, however, whether this holds for boxes.

For more examples, we refer the reader to \cite{webpageNicolau}. A particularly interesting example is the $4 \times 4 \times 4$ box, which has 5051532105 tilings, divided into 93 flip connected components. The largest connected component has zero twist and 4412646453 tilings; and the values of the twist range from $-4$ to $4$. 

\bibliography{biblio}{}
\bibliographystyle{plain}

\noindent
\footnotesize
Departamento de Matem\'atica, PUC-Rio \\
Rua Marqu\^es de S\~ao Vicente, 225, Rio de Janeiro, RJ 22451-900, Brazil \\
\url{milet@mat.puc-rio.br}\\
\url{saldanha@puc-rio.br}

\end{document}